\documentclass[a4paper,10pt]{article}
\usepackage{amsthm,amsmath,amssymb,bbm,enumerate,nicefrac,hyperref}
\usepackage[a4paper, left=1.5cm, right=1.5cm]{geometry}
\usepackage{color}
\usepackage{xfrac}
\usepackage{comment}

\hypersetup{colorlinks=true}

\usepackage[capitalize]{cleveref}
\crefname{subsection}{Subsection}{Subsections}

\newtheorem{theorem}{Theorem}[section]
\newtheorem{lemma}[theorem]{Lemma}
\newtheorem{proposition}[theorem]{Proposition} 
 
\newtheorem{setting}[theorem]{Setting} 

\newcommand{\vertiii}[1]{{\left\vert\kern-0.25ex\left\vert\kern-0.25ex\left\vert #1 
    \right\vert\kern-0.25ex\right\vert\kern-0.25ex\right\vert}}
    
\newcommand{\vertiiii}[1]{{|\kern-0.25ex|\kern-0.25ex| #1 
    |\kern-0.25ex|\kern-0.25ex|}}

\newcommand{\R}{\mathbb{R}}
\newcommand{\N}{\mathbb{N}}
\newcommand{\E}{\mathbb{E}}
\newcommand{\Z}{\mathbb{Z}}
\renewcommand{\P}{\mathbb{P}}

\title{Full history recursive multilevel Picard \\ approximations for ordinary differential \\ equations with expectations}
\author{ 
	Christian Beck$^{1}$, 
	Martin Hutzenthaler$^2$, 
	\\
	Arnulf Jentzen$^{3}$,  
	and 
	Emilia Magnani$^{4}$
	\bigskip
	\\
	\small{$^1$ Department of Mathematics, ETH Zurich, Z\"urich, Switzerland;} \\
	\small{  Faculty of Mathematics and Computer Science, University of M\"unster, }\\
	\small{M\"unster, Germany; e-mail: \texttt{christian.beck@uni-muenster.de}}
	\\
	\small{$^2$ Faculty of Mathematics, University of Duisburg-Essen,}\\
	\small{Essen, Germany; e-mail: \texttt{martin.hutzenthaler@uni-due.de}}
	\\
	\small{$^3$ Department of Mathematics, ETH Zurich, Z\"urich, Switzerland;}\\
	\small{ Faculty of Mathematics and Computer Science, University of} 
	\\
	\small{ M\"unster, M\"unster, Germany; School of Data Science and Shenzhen}
	\\ 
	\small{ Research Institute of Big Data, The Chinese  University of}\\
	\small{ Hong Kong, Shenzhen, China; e-mail: \texttt{ajentzen@uni-muenster.de}}
	\\
	\small{$^4$ Department of Computer Science, University of T\"ubingen,}\\
	\small{T\"ubingen, Germany; 
	Department of Mathematics,  ETH Zurich,}  \\
	\small{ Z\"urich, Switzerland; e-mail: \texttt{emilia.magnani@uni-tuebingen.de }}
}

\begin{document}
\maketitle

\begin{abstract}
We consider ordinary differential equations (ODEs) which involve
expectations of a random variable. These ODEs are special cases of McKean--Vlasov stochastic differential equations (SDEs).  A plain vanilla Monte Carlo approximation method for such ODEs requires a computational cost of order $\varepsilon^{-3}$ to achieve
a root-mean-square error of size $\varepsilon$. In this work we adapt recently introduced full history recursive multilevel Picard (MLP) algorithms to reduce this computational complexity. Our main result shows  for every $\delta>0$ that the proposed MLP approximation algorithm requires only a computational effort of order $\varepsilon^{-(2+\delta)}$ to achieve a root-mean-square error of size $\varepsilon$.
\end{abstract}

\newpage

\tableofcontents

\section{Introduction}

It is a very challenging task in applied mathematics to design efficient approximation algorithms for high-dimensional partial differential equations (PDEs).  
Recently, significant progress has been made in this area of research through the development of so-called \textit{full history recursive multilevel Picard} (MLP) approximation algorithms \cite{hutzenthaler2016multilevel,MR3946468,hutzenthaler2018overcoming}. 
Up to now, MLP approximation algorithms are the only approximation algorithms in the scientific literature for which it has been rigorously proved that they can overcome the curse of dimensionality in the numerical approximation of second-order semilinear elliptic and parabolic PDEs with general time horizons in the sense that the number of computational operations needed to achieve a desired approximation accuracy $\varepsilon \in (0,\infty)$ grows at most polynomially in both the PDE dimension $d \in \N = \lbrace 1,2,3,\dots \rbrace$ and the reciprocal $\nicefrac{1}{\varepsilon}$ of the desired  approximation accuracy $\varepsilon$.
We also refer to \cite{beck2020overcoming,beckAllenCahn,becker2020numerical, giles2019generalised,hutz2019overcoming,hutzenthaler2020multilevel,hutzenthaler2018overcoming,hutzenthaler2019overcoming,hutzenthaler2020multi} and the overview articles \cite{EMSNewsletter2020OverviewOnDLBasedApproximationMethodsForPDEs,EHanJentzen2020OverviewFromNonlinearMonteCarloToMachineLearning} for computational problems in which MLP approximation algorithms have been shown to overcome the curse of dimensionality. \par
To develop a better understanding for the recently proposed MLP approximation methods, we aim within this article to extend the MLP approximation algorithms to the case of ordinary differential equations (ODEs) involving the expectations of random variables. 
To achieve a root-mean-square error of size $\varepsilon\in(0,\infty)$ a plain vanilla Monte Carlo method requires a computational cost of  order $\varepsilon^{-3}$.
We reduce this computational cost and show
that for an arbitrarily small $\delta\in(0,\infty)$ the proposed MLP approximation schemes achieve an approximation accuracy of size $\varepsilon\in(0,\infty)$ with a computational effort of order $\varepsilon^{-(2+\delta)}$.
More precisely, \cref{THEOREM} below, the main result of this article, proves under suitable assumptions that for 
every $\varepsilon \in (0,1]$, $\delta \in (0,\infty)$ the  MLP approximation scheme achieves a root mean
 square error of at most $\varepsilon$ with a computational effort of order $\varepsilon^{-(2+\delta)}$.
As an illustration of 
\cref{THEOREM} below, we present now in the following result,  \cref{theoremintroduction}, a simplified version of \cref{THEOREM}.

\begin{theorem} \label{theoremintroduction}
Let $d \in \N$, $\delta \in (0, \infty)$,
$T, L \in [0,\infty)$, $\Theta = \cup_{n=1}^{\infty} {\Z}^n$, $X \in \mathcal{C}([0,T],\R^d)$, $\xi \in \R^d$, let $ \left\|  \cdot \right\| \! \colon \R^d \to [0,\infty) $ be a norm on $\R^d$, 
  let $(S,\mathcal{S})$ be a
measurable space, 
let $F \colon{\R}^d\times
S\rightarrow {\R}^d$  be $(\mathcal{B}({\R}^d) \otimes \mathcal{S}) / \mathcal{B}({\R}^d)$-measurable,
assume for all $x,y \in {\R}^d$, $s \in S$  that
 $\| F(x,s)-F(y,s)\| \leq L\|x-y\|$,
 let $(\Omega,\mathcal{F},{\P})$ be a probability space,
 let $Z^\theta \colon\Omega\rightarrow S$, $\theta \in\Theta$,
be i.i.d.\ random variables, 
assume that $\E[ \|F(\xi, Z^0) \|^2] < \infty$,
assume for all $t \in [0,T]$ that 
 $X(t)=\xi + \int_0^t \E[F(X(r),Z^0)]\,dr$,
let $\mathfrak{r}^\theta \colon \Omega \rightarrow [0,1]$,
$\theta \in \Theta$, be independent $\mathcal{U}_{[0,1]}$-distributed random variables,
assume that
${(\mathfrak{r}^\theta)}_{\theta\in\Theta}$ and
${(Z^\theta)}_{\theta\in\Theta}$ are independent,
let ${\mathcal{X}}_{n,m}^{\theta} \colon [0,T]\times\Omega \rightarrow {\R}^d $, $n,m \in \N_0$, $\theta \in \Theta$, satisfy
for all  $ n \in {\N_0}$, $m \in \N$, $\theta \in \Theta$, $t \in [0,T]$ 
that 
\begin{equation}\label{approxscheme}
 \begin{split}
 {\mathcal{X}}_{n,m}^{\theta} (t) 
 & = \Bigg[
 \sum_{l=1}^{n-1}
 \frac{t}{m^{n-l}}  \sum_{k=1}^{m^{n-l}} F \big( \mathcal{X}_{l,m}^{(\theta,l,k)} (\mathfrak{r}^{(\theta,l,k)}t), Z^{(\theta,l,k)} \big) 
 - F \big( \mathcal{X}_{l-1,m}^{(\theta,l,-k)}(\mathfrak{r}^{(\theta,l,k)}t), Z^{(\theta,l,k)} \big)\Bigg] 
 \\& 
 + \Bigg[  \frac{t \mathbbm{1}_{\N}(n)}{m^n} \sum_{k=1}^{m^n} F \big(\xi,Z^{(\theta,0,k)}\big)\Bigg]  + \xi,
 \end{split}
\end{equation}
and for every $n, m \in {\N}$
let $RV_{n,m} \in \N$ be the number of realizations of random
variables $(Z^\theta)_{\theta \in \Theta}$ which are used to compute one realization of
${\mathcal{X}}_{n,m}^{0} (T) $
(cf.\ \eqref{againRV}).
Then
there exist $c \in  \R$ and $N=(N_{\varepsilon})_{\varepsilon\in (0,1]} \colon  (0, 1]  \rightarrow \N$
such that for all $\varepsilon \in (0, 1]$ it holds that
$RV_{N_{\varepsilon},N_{\varepsilon}} \leq c\,  \varepsilon^{-(2+\delta)}$ and
$( {\E} [\| X(T) - {\mathcal{X}}_{N_{\varepsilon},N_{\varepsilon}}^{0}(T)\|^2] )^{1/2} \leq \varepsilon$. 
\end{theorem}

\cref{theoremintroduction} is an immediate consequence of \cref{THEOREM} in \cref{section3} below.
\cref{theoremintroduction} establishes under suitable conditions that for every $\delta \in (0,\infty)$ there exists $c \in \R$ such 
that the solution $X \in \mathcal{C}([0,T],\R^d)$ of the differential equation $X(t)=\xi + \int_0^t \E[F(X(r),Z^0)]\,dr$, $t \in [0,T]$, (cf.\ \cref{finiteexp}), can be approximated by the recursive MLP approximation schemes in \eqref{approxscheme} with a root mean square 
 error of size $\varepsilon \in (0,1]$ and a computational effort that is bounded by $c\,\varepsilon^{-(2+\delta)}$.
The computational effort is quantified by the numbers $RV_{n,m}$, $n, m \in \N$. 
The function $F \colon{\R}^d\times S\rightarrow {\R}^d$ is required to be $(\mathcal{B}({\R}^d) \otimes \mathcal{S}) / \mathcal{B}({\R}^d)$
-measurable and Lipschitz continuous in the first variable, uniformly in the second variable,
and we assume that $\E[ \|F(\xi, Z^0) \|^2] < \infty$. 
We note that differential equations of the type $X(t)=\xi + \int_0^t \E[F(X(r),Z^0)]\,dr$, $t \in [0,T]$, can be considered as a special case of 
so-called \textit{McKean--Vlasov stochastic differential equations} (SDEs). 
Numerical approximation methods for McKean--Vlasov SDEs have been widely developed in the scientific literature. In particular, in the article 
 \cite{bossy1997stochastic} Bossy and Talay carried out the approximation of McKean--Vlasov SDEs with simulations through interacting particles systems and time discretizations. Many other authors have contributed to this approach. In particular, we refer among others to \cite{antonelli2002,belomestny2018projected,belomestny2019iterative, bossy2002rate,bossy1996convergence,DosReis10.1093/imanum/draa099,reis2018importance,reisinger2020adaptive, szpruch2019, talay2003stochastic}.
Moreover, alternative approximation methods for McKean--Vlasov SDEs relying on cubature formulas \cite{CHAUDRUDERAYNAL20152206,crisan2019cubature}, analytical expansions \cite{gobet2018analytical}, or tamed Milstein schemes \cite{bao2021first,kumar2020explicit} have been developed.
For further numerical approximation methods for McKean--Vlasov SDEs we also refer, e.g., to \cite{agarwal2020fourier, chassagneux2019,biswas2020well}. 
The problems which are treated in these references are of course far more general and involved than the expectation ODEs which we consider in this article. 
To the best of our knowledge, \cref{theoremintroduction} is the first result in the scientific literature which shows that solutions of the special class of McKean--Vlasov SDEs considered in this article can be approximated with a root-mean-square error of size $\varepsilon$
with a computational cost of order $\varepsilon^{-(2+\delta)}$.  \par
The remainder of this article is organized as follows.
\cref{Section2} introduces multilevel Picard (MLP) approximation schemes for a special type of ODEs involving expectations of
 a random variable (see \cref{setting} in \cref{subsectionsetting}) and provides $L^2$-error estimates for the differences
  between the MLP approximations and the exact solution of the ODE under consideration. 
In \cref{section3} we combine the error analysis from \cref{Section2} with suitable estimates for the computational effort for the
 proposed MLP approximation algorithms (see \cref{lemma1} below) to perform a complexity analysis for the proposed MLP
  approximation algorithms and we show in \cref{THEOREM} that for every $ \delta \in (0,\infty)$, $\varepsilon \in (0,1]$ it 
  holds that MLP approximations can achieve a root mean square error of size $\varepsilon \in (0,1]$ with a computational 
  effort of size
$\varepsilon^{-(2+\delta)}$.

\section{Error analysis for multilevel Picard (MLP) approximations}\label{Section2}

In this section we introduce in \cref{setting} in \cref{subsectionsetting} below MLP approximations (see \eqref{numsol} in \cref{setting}) for ODEs involving expectations (see \eqref{sol} in \cref{setting}) and we provide in \cref{thmerrorestimate} in \cref{subsectionerrorestimate} an $L^2$-error analysis for the proposed MLP approximation schemes. 
In particular, we establish in \cref{thmerrorestimate}, the main result of this section, an upper bound for the root mean square error between solutions of ODEs involving expectations and the corresponding MLP approximations. 

Our proof of \cref{thmerrorestimate} is heavily motivated by \cite{hutzenthaler2018overcoming} and has, rougly speaking, four steps: 
first, (i) we perform a bias-variance decomposition of the mean square error between solutions of ODEs involving expectations and the corresponding MLP approximations, 
second, (ii) we estimate the biases of the proposed MLP approximations, 
third, (iii) we estimate the variances of the proposed MLP approximations, 
and, finally, (iv) we obtain the final bound by applying a Gronwall type argument to the recursive inequalities obtained through combining the bias-variance decomposition (see (i) above) with the bias estimates (see (ii) above) and the variance estimates (see (iii) above). 
Our proof of \cref{thmerrorestimate} relies on several auxiliary results, which we present in \crefrange{subsectiongronwall}{subsectionerrorestimate} below. 
In \cref{subsectiongronwall} we recall in \cref{cor2.2} and \cref{l3.12} elementary and well-known time-discrete Gronwall inequalities. 
Proofs for the results in \cref{cor2.2} and \cref{l3.12} in \cref{subsectiongronwall} below can be found, e.g., in \cite[Corollary 2.2 and Lemma 3.12]{hutzenthaler2019overcoming}.
In \cref{subsectiononrando} we exhibit in \crefrange{l2.14}{l2.16new} elementary and well-known results about random variables which arise from evaluating random fields at random indices.
Proofs for \cref{l2.14} and \cref{l2.15} can, e.g., be found in  \cite[Lemma 2.14]{hutzenthaler2019overcoming} and \cite[Lemma 2.3]{hutzenthaler2018overcoming}. We also include in this section a proof of \cref{l2.16new}, which is a slight modification of \cite[Lemma 2.16]{hutzenthaler2019overcoming}. 
In \cref{subsectionapriori} we establish in \cref{l3.3} elementary a priori bounds for solutions of ODEs involving expectations.
In \cref{subsectionproperties} we present in \cref{propMLP} and \cref{expMLP} some fundamental measurability, integrability, and distribution properties of MLP approximations. The proofs of \cref{propMLP} and \cref{expMLP} employ the well-known and elementary results in \cref{l3.4}, \cref{l3.5}, and \cref{l3.7} below. Proofs for \cref{l3.4} and \cref{l3.5} can be found, e.g., in \cite[Lemma 3.4 and Lemma 3.5]{hutzenthaler2019overcoming}. Note that \cref{l3.7} is a slightly modified version of 
\cite[Lemma 3.7]{hutzenthaler2019overcoming}. 
In \cref{subsectionerrorestimate} we recall some elementary and well-known results in \cref{variances}--\cref{l3.14} below before proving \cref{thmerrorestimate}. We refer, e.g., to \cite[Lemma 3.10]{hutzenthaler2019overcoming} and \cite[Lemma 2.9]{hutzenthaler2019overcoming} for proofs of \cref{variances} and \cref{l2.9}.  
\cref{l3.14}, which is a simple consequence of Tonelli's theorem, is a slightly modified version of \cite[Lemma 3.14]{hutzenthaler2019overcoming}.

\subsection{Setting}\label{subsectionsetting}
\begin{setting} \label{setting}
Let $d \in {\N}$, $T, L \in [0,\infty)$, $\Theta = \cup_{n=1}^{\infty} {\Z}^n$, $ X \in \mathcal{C}([0,T], {\R}^d)$, $\xi \in \R^d$, 
let $ \left\|  \cdot \right\| \! \colon \R^d \to [0,\infty) $ be the standard norm on $\R^d$, 
let $(S,\mathcal{S})$ be a
measurable space, let $F\colon{\R}^d\times
S\rightarrow {\R}^d$ be $(\mathcal{B}({\R}^d) \otimes \mathcal{S}) / \mathcal{B}({\R}^d)$-measurable,
assume for all $x,y \in {\R}^d$, $s \in S$ that 
\begin{equation} \label{lipsch}
\left\| F(x,s)-F(y,s)  \right\| \leq L \| x-y \|,
\end{equation}
let $(\Omega,\mathcal{F},{\P})$ be a probability space,
let $Z^\theta \colon\Omega\rightarrow S$, $\theta \in\Theta$,
be independent and identically distributed (i.i.d.) random variables, 
let $\mathfrak{r}^\theta \colon \Omega \rightarrow [0,1]$,
$\theta \in \Theta$, be independent $\mathcal{U}_{[0,1]}$-distributed random variables,
 let $\mathcal{R}^\theta = (\mathcal{R}^\theta_t)_{t \in [0,T]} \colon
[0,T]\times\Omega \rightarrow [0,T]$, $\theta \in \Theta$, satisfy for all $t \in [0,T]$,  $\theta \in \Theta$ that 
$\mathcal{R}^\theta_t =  \mathfrak{r}^\theta t$,
assume that
${(\mathfrak{r}^\theta)}_{\theta\in\Theta}$ and
${(Z^\theta)}_{\theta\in\Theta}$ are independent,
assume  for all $t \in [0,T]$ that $ \int_0^t  \E [ \| F (  X(r), Z^0 ) \| ]\,  dr  <\infty$ and 
\begin{equation} \label{sol}
X(t) = \xi + \int_0^t {\E [ F (  X(r), Z^0 )]}\, dr,
\end{equation}
and let  
${\mathcal{X}}_{n,m}^\theta \colon [0,T]\times\Omega \rightarrow {\R}^d $, $n,m \in \N_0$, $\theta \in \Theta$, satisfy
for all  $n \in \N_0$, $m \in {\N}$, $\theta \in \Theta$, $t \in [0,T]$ 
that 
\begin{equation} \label{numsol}
 \begin{split}
{\mathcal{X}}_{n,m}^\theta (t) & = \Bigg[
 \sum_{l=1}^{n-1}
\frac{t}{m^{n-l}}  \sum_{k=1}^{m^{n-l}} \bigg( F\big( \mathcal{X}_{l,m}^{(\theta,l,k)} (\mathcal{R}_{t}^{(\theta,l,k)}), Z^{(\theta,l,k)} \big) - F\big( \mathcal{X}_{l-1,m}^{(\theta,l,-k)}(\mathcal{R}_{t}^{(\theta,l,k)}), Z^{(\theta,l,k)} \big) \bigg)\Bigg]
\\
& + \left[  \frac{t \mathbbm{1}_{\N}(n)}{m^n}  \sum_{k=1}^{m^n} F\big(\xi,Z^{(\theta,0,k)}\big)\right]  + \xi.
 \end{split}
\end{equation}
\end{setting}

\subsection{Time-discrete Gronwall inequalities}\label{subsectiongronwall}

\begin{lemma}\label{cor2.2}
Let $N \in \N \cup \lbrace \infty \rbrace$, $\alpha,\beta \in [0,\infty)$, $(\epsilon_n)_{n \in \N_0\cap [0,N]} \subseteq[0,\infty]$ satisfy for all $n \in \N_0\cap [0,N]$ that
$\epsilon_n \leq \alpha + \beta \sum_{k=0}^{n-1} \epsilon_k $.
Then it holds for all $n \in \N_0\cap [0,N]$ that
$\epsilon_n \leq \alpha(1+\beta)^n \leq \alpha e^{\beta n} < \infty$.
\end{lemma}

\begin{lemma} \label{l3.12}
Let $\alpha,\beta \in [0,\infty)$, $M \in (0,\infty)$, $(\epsilon_{n,k})_{n,k \in \N_0} \subseteq[0,\infty]$ satisfy for all
$n,k \in \N_0$ that
$\epsilon_{n,k} \leq
\frac{\alpha}{M^{n+k}} + \beta \sum_{l=0}^{n-1}
\frac{\epsilon_{l,k+1}}{M^{n-(l+1)}} $.
Then it holds for all $n,k \in \N_0$ that
$\epsilon_{n,k} \leq
\frac{\alpha(1+\beta)^n}{M^{n+k}} < \infty$.
\end{lemma}

\subsection{On random evaluations of random fields}\label{subsectiononrando}

\begin{lemma}\label{l2.14}
Let $(\Omega, \mathcal{F})$, $(S, \mathcal{S})$,
$(E, \mathcal{E})$ be measurable spaces, let $ U = 
(U(s))_{s \in S} = (U(s,\omega))_{(s,\omega) \in S \times \Omega} \colon S \times \Omega \rightarrow E$
be $(\mathcal{S} \otimes \mathcal{F}) / \mathcal{E}$-measurable, and let $X \colon
\Omega \rightarrow S$ be  $\mathcal{F}/\mathcal{S}$-measurable. 
Then it holds that $U(X) = (U(X(\omega),\omega))_{\omega \in \Omega}$
$\colon \Omega \rightarrow E$ is
$\mathcal{F}/\mathcal{E}$-measurable.  
\end{lemma}

\begin{lemma} \label{l2.15}
Let $(\Omega, \mathcal{F}, \P)$ be a probability space, let $(S,\delta)$ be a separable metric space, let $ U = (U(s))_{s \in S}=(U(s,\omega))_{(s,\omega) \in S \times \Omega}
 \colon S \times \Omega \rightarrow [0,\infty)$ be a continuous
  random field, let $X \colon \Omega \rightarrow S$ be a random
  variable, and  assume that $U$ \!and $X$ are independent.
  Then it holds that $U(X) = (\Omega \ni \omega \mapsto U(X(\omega),\omega) \in [0,\infty))$ is $\mathcal{F}/\mathcal{B}([0,\infty))$-measurable and 
  $\E[U(X)] = \int_S \E[U(s)] (X(\P)_{\mathcal{B}(S)}) (ds)$.
\end{lemma}

\begin{lemma} \label{l2.16new}
Let $d \in \N$, let $(\Omega, \mathcal{F}, \P)$ be a probability space,
let $ \left\|\cdot\right\|\! \colon \R^d \to
[0,\infty)$ 
be a norm on $\R^d$,
let $(S,\delta)$ be a separable metric space,
let $
U = (U(s))_{s \in S}
=(U(s,\omega))_{(s,\omega) \in S \times \Omega}
\colon S \times \Omega \rightarrow \R^d$
be a continuous
random field, let $X \colon \Omega \rightarrow S$ be a random variable, assume that $U$ and $X$ are independent,
and assume that $\int_S \E[\| U(s) \|](X(\P)_{\mathcal{B}(S)})(ds) < \infty$.
 Then 
 \begin{enumerate}[(i)]
 \item\label{it:l2.16new:item1} it holds that 
 $U(X) = ( \Omega \ni \omega \mapsto U(X(\omega), \omega ) \in \R^d)$ is $\mathcal{F}/\mathcal{B}(\R^d)$-measurable,
 \item\label{it:l2.16new:item2} it holds that $(X(\P)_{\mathcal{B}(S)})
 (\lbrace s \in S \colon \E[\| U(s) \|] = \infty  \rbrace)=0$, 
 and
 \item\label{it:l2.16new:item3} it holds that $\E[\| U(X) \|] < \infty$
  and 
	$ \E[U(X)] = \int_S \E[U(s)] (X(\P)_{\mathcal{B}(S)}) (ds) $.
  \end{enumerate}
\end{lemma}
\begin{proof}[Proof of \cref{l2.16new}]
Throughout this proof
let $ \vertiii \cdot  \colon \R^d \to [0,\infty) $ be the standard norm on $\R^d$,
let $Z \colon \Omega \rightarrow \Omega$ satisfy for all $\omega \in \Omega$ that 
$Z(\omega)= \omega$,
let $u_k=(u_k(s))_{s \in S}=(u_k(s,\omega))_{(s, \omega) \in S \times \Omega}
\colon S \times \Omega \rightarrow \R$, $k \in \lbrace 1,2, \dots , d \rbrace$, satisfy for all $s \in S$, 
 $\omega \in \Omega$ that 
$ (u_1(s,\omega),u_2(s,\omega), \dots , u_d(s,\omega)) = U(s,\omega)$, 
and let $\mathbbmtt{u}\hspace{.04cm}_i \colon S \times \Omega \rightarrow [0,\infty)$, $i \in \lbrace 1,2, \dots , d \rbrace$, and $\mathfrak{u}\hspace{.01cm}_i \colon S \times \Omega \rightarrow [0,\infty)$, $i \in \lbrace 1,2, \dots , d \rbrace$, satisfy for all $s \in S$,
$\omega \in \Omega$, $i \in \lbrace 1,2, \dots , d \rbrace$ that
$\mathbbmtt{u}\hspace{.04cm}_i(s,\omega) = \max \lbrace u_i(s,\omega),0 \rbrace$ 
and 
$\mathfrak{u}\hspace{.01cm}_i(s,\omega) = \max \lbrace -u_i(s,\omega),0 \rbrace$.
Observe that for all $i \in \lbrace 1,2, \dots , d \rbrace$ it holds that $u_i = \mathbbmtt{u}\hspace{.04cm}_i - \mathfrak{u}\hspace{.01cm}_i$.
Moreover, note that the hypothesis that $U$ is a continuous
random field assures that $U$ is $\big(\mathcal{B}(S) 
\otimes \sigma_\Omega(\lbrace U(s)\colon s \in S\rbrace)\big) / \mathcal{B}(\R^d)$-measurable.
Next note that the fact that $\sigma_\Omega(\lbrace U(s)\colon s \in S\rbrace) \subseteq
\mathcal{F}$ ensures that $Z$ is $\mathcal{F}/\sigma_\Omega(\lbrace U(s)\colon s \in S\rbrace)$-measurable.
The hypothesis that $X$ is
$\mathcal{F}/\mathcal{B}(S)$-measurable hence proves that $\Omega \ni \omega \mapsto (X(\omega), Z(\omega))$ $=
(X(\omega),\omega) \in S \times\Omega$ is $\mathcal{F}/ \big(\mathcal{B}(S) \otimes \sigma_\Omega(\lbrace U(s)\colon s \in S\rbrace) \big)$-measurable. Combining this with the fact that $U$ is 
$\big(\mathcal{B}(S) 
\otimes \sigma_\Omega(\lbrace U(s)\colon s \in S\rbrace)\big) / \mathcal{B}(\R^d)$-measurable establishes item~\eqref{it:l2.16new:item1}.
Furthermore, observe that the hypothesis that
$\int_S \E[\| U(s)\|] (X(\P)_{\mathcal{B}(S)})$ $(ds)< \infty$ ensures that
\begin{equation}
   (X(\P)_{\mathcal{B}(S)})
 (\lbrace s \in S \colon \E[\| U(s) \|] = \infty  \rbrace)=0. 
\end{equation}
This establishes item~\eqref{it:l2.16new:item2}.
Moreover, note that \cref{l2.15} and the hypothesis that  
$\int_S \E[\| U(s)\|] (X(\P)_{\mathcal{B}(S)})(ds) < \infty$
assure that
\begin{equation}\label{meannormfinite}
    \E[\| U(X) \|] = \int_S \E[\| U(s) \|] (X(\P)_{\mathcal{B}(S)})(ds) < \infty.
\end{equation}
Next let $a \in (0, \infty)$ satisfy that
\begin{equation}\label{akreys}
    \vertiii{ U(X)} \leq a \| U(X) \|
\end{equation}
(cf., e.g.,
Kreyszig 
\cite[Theorem 2.4-5]{MR0467220}).
Observe that \eqref{akreys}, \eqref{meannormfinite}, the fact that for all $i \in \lbrace 1,2, \dots , d \rbrace$ it holds that $|u_i| = \mathbbmtt{u}\hspace{.04cm}_i + \mathfrak{u}\hspace{.01cm}_i$, and the fact that for all $i \in \lbrace 1,2, \dots , d \rbrace$ it holds that
$|u_i(X)| \leq \vertiii{ U(X)}$ imply that for all $i \in \lbrace 1,2, \dots , d \rbrace$ it holds that
\begin{equation}\label{newforHutz}
  \E[\mathbbmtt{u}\hspace{.04cm}_i(X)] + \E[\mathfrak{u}\hspace{.01cm}_i(X)]
  = \E[|u_i(X)|] \leq \E[ \vertiii{U(X)} ]
  \leq a \E[\| U(X) \|] < \infty.
\end{equation}
This and the hypothesis that  
$\int_S \E[\| U(s)\|] (X(\P)_{\mathcal{B}(S)})(ds) < \infty$ ensure that for all $i \in \lbrace 1,2, \dots , d \rbrace$ it holds that
\begin{equation}
    \int_S (\E[ \mathbbmtt{u}\hspace{.04cm}_i(s)
    + \mathfrak{u}\hspace{.01cm}_i(s)])
    (X(\P)_{\mathcal{B}(S)})(ds)
    \leq a \int_S (\E[\|U(s)\|] )
    (X(\P)_{\mathcal{B}(S)})(ds) < \infty.
\end{equation}
\cref{l2.15}, \eqref{newforHutz}, and the fact that for all $i \in \lbrace 1,2, \dots , d \rbrace$ it holds that  $u_i = \mathbbmtt{u}\hspace{.04cm}_i - \mathfrak{u}\hspace{.01cm}_i$ hence demonstrate that for all $i \in \lbrace 1,2, \dots , d \rbrace$ it holds that
\begin{equation}
\begin{split}
    \E[u_i(X)] 
    &= \E[\mathbbmtt{u}\hspace{.04cm}_i(X) - \mathfrak{u}\hspace{.01cm}_i(X)] =\E[\mathbbmtt{u}\hspace{.04cm}_i(X)] -  \E[\mathfrak{u}\hspace{.01cm}_i(X)]
    = \int_S \E[ \mathbbmtt{u}\hspace{.04cm}_i(s)] (X(\P)_{\mathcal{B}(S)})(ds)
    - \int_S \E[ \mathfrak{u}\hspace{.01cm}_i(s)] (X(\P)_{\mathcal{B}(S)})(ds)
    \\
    & = \int_S (\E[ \mathbbmtt{u}\hspace{.04cm}_i(s)]
    -\E[ \mathfrak{u}\hspace{.01cm}_i(s)])
    (X(\P)_{\mathcal{B}(S)})(ds)
     = \int_S \E[ u_i(s)] (X(\P)_{\mathcal{B}(S)})(ds).
    \end{split}
\end{equation}
This implies that 
\begin{equation}
      \E[U(X)] = \int_S \E[U(s)] (X(\P)_{\mathcal{B}(S)})
      (ds).
  \end{equation}
  Combining this and \eqref{meannormfinite} establishes
  item~\eqref{it:l2.16new:item3}.
  The proof of \cref{l2.16new} is thus completed.
\end{proof}

\subsection{A priori bounds for solutions of ordinary differential equations}\label{subsectionapriori}

\begin{lemma} \label{l3.3}
Assume \cref{setting}.
Then it holds for all $t \in [0,T]$ that 
\begin{equation}\label{15stat}
 \| X(t) - \xi \| \leq  T
 \big( \E  \big[  \| F(\xi, Z^0) \|^2 \big] \big)^{1/2} e^{L T}.
\end{equation}
\end{lemma}

\begin{proof}[Proof of \cref{l3.3}]
Throughout this proof assume without loss of generality (w.l.o.g.) that  $ \E[\| F(\xi, Z^0) \|^2] $ $<\infty$. 
Observe that \eqref{sol} ensures that for all $t \in [0,T]$ it holds that
\begin{equation} \label{eq1l}
 \| X(t) - \xi \| =  \left\| \int_0^t {{\E}\big[ F (  X(s),
 Z^0 )  \big]}\,  ds \right\| \leq   \int_0^t { \left\| \E\big[ F (  X(s), Z^0)  \big] \right\|}\,  ds\\
 \leq 
  \int_0^t {{\E}\big[ \| F(  X(s), Z^0 )  \| \big]}\, ds.
\end{equation}
In addition, observe that the hypothesis that $ X \in \mathcal{C}([0,T], {\R}^d)$ ensures that
\begin{equation}\label{pergron}
    \int_0^T \|X(s) -\xi \| \, ds < \infty.
\end{equation}
Moreover, note that \eqref{lipsch} and the triangle inequality ensure that for all  $x \in {\R}^d$, $s \in S$ it holds that
\begin{equation} 
\| F(x,s) \|  \leq \| F(\xi,s) \| + \| F(x,s) -F(\xi,s) \| \leq 
 \| F(\xi,s) \| + L \| x-\xi \|.
\end{equation}
This, \eqref{eq1l}, and Jensen's inequality 
 ensure that for all $t \in [0,T]$ it holds that
\begin{equation}
\begin{split}
\| X(t) - \xi \| 
& \leq   \int_0^t {  \E\big[\| F (  X(s), Z^0 ) \| \big] }\, ds 
\leq   \int_0^t {\E} \big[ \|  F(\xi,Z^0) \| + L \|X(s)-\xi \|  \big]\,ds  \\
& = \int_0^t \E \big[ \| F(\xi,Z^0) \|\big]\, ds  
 + L\! \int_0^t  \| X(s) - \xi \| \,ds 
 \leq T \E \big[  \| F(\xi,Z^0) \|  \big]  
 + L\! \int_0^t  \| X(s) - \xi \| \, ds\\
 &= T \big(\big(\E \big[  \| F(\xi,Z^0) \|  \big] \big)^2 \big)^{1/2} 
 + L\! \int_0^t  \| X(s) - \xi \| \,ds 
 \leq T {\big(\E \big[  \| F(\xi,Z^0) \| ^2 \big] \big)}^{1/2} 
 + L\!   \int_0^t  \| X(s) - \xi \| \, ds.
 \end{split}
\end{equation}
Combining this and \eqref{pergron} with Gronwall's integral inequality implies that for all $t \in [0,T]$ it holds that
\begin{equation}
 \| X(t) - \xi \|   \leq T {\big(\E \big[  \| F(\xi,Z^0) \| ^2 \big] \big)}^{1/2} e^{   L t }
   \leq T {\big(\E \big[\| F(\xi,Z^0) \|^2 \big] \big) }^{1/2} e^{ L T}.
\end{equation}
This establishes \eqref{15stat}.
The proof of \cref{l3.3} is thus completed.
\end{proof}

\subsection{Properties of MLP approximations}\label{subsectionproperties}

\begin{lemma}\label{l3.4}
Let $d,N \in \N$, let $(\Omega, \mathcal{F}, \P)$ be a probability space, let $X_k \colon \Omega \rightarrow \R^d$, $k \in \lbrace1,2, \dots ,N \rbrace$, be independent random variables, let
$Y_k \colon \Omega \rightarrow \R^d$, $k \in \lbrace1,2, \dots , N \rbrace$, be independent random variables, and assume for every
$k \in \lbrace1,2, \dots , $ $N \rbrace$ that $X_k$ and $Y_k$ are identically distributed.
Then it holds that $\big( \sum_{k=1}^N X_k \big) \colon \Omega \rightarrow \R^d$ and $\big( \sum_{k=1}^N Y_k \big) \colon \Omega \rightarrow \R^d$ are identically distributed random variables.
\end{lemma}

\begin{lemma} \label{l3.5}
Let $(\Omega, \mathcal{F}, \P)$ be a probability space, let $(S,\delta)$ be a separable metric space, let $(E,\mathfrak{d})$ be a metric space, let    $U,V \colon S\times \Omega \rightarrow E$  be
 continuous random fields, let $X,Y \colon \Omega \rightarrow S$ be random variables, assume that $U$ and $X$ are independent, assume that\   $V$ and $Y$ are independent, assume for all $ s \in S$ that  $U(s)$ and $V(s)$ are identically distributed, and assume that $X$ and $Y$ are identically distributed. Then it holds that $U(X)=(U(X(\omega),\omega))_{\omega \in \Omega} \colon \Omega \rightarrow E$ and 
 $V(Y)=(V(Y(\omega),\omega))_{\omega \in \Omega} \colon \Omega \rightarrow E$ are identically distributed random variables.
\end{lemma}

\begin{lemma}[Properties of MLP approximations]
\label{propMLP}   
Assume \cref{setting} and let $m \in \N$. Then
\begin{enumerate}[(i)]
\item\label{it:propMLP:item1} for all $\theta \in \Theta$, $n \in \N_0$ it holds that
		${\mathcal{X}}_{n,m}^\theta \colon [0,T]\times\Omega \rightarrow {\R}^d $ is a stochastic process with continuous sample paths,
\item\label{it:propMLP:item2} for all $\theta \in \Theta$, $n \in \N_0$ it holds that
		${\mathcal{X}}_{n,m}^\theta $ is 
		$\big( \mathcal{B}([0,T])\otimes\sigma_\Omega
		({(\mathfrak{r}^{(\theta,\vartheta)})}_{\vartheta \in \Theta},{(Z^{(\theta,\vartheta)})}_{\vartheta \in \Theta})\big)$ $/\mathcal{B}(\R^d)$-measurable,
\item\label{it:propMLP:item3} for all  $\theta \in \Theta$, $n \in \N$, $t \in [0,T]$
it holds that
\begin{equation}
\begin{split}
 \left((\N_0\cap[0,n-1])\times\N\right) \ni (l,k)
  \mapsto {\begin{cases}
    F\big(\xi, Z^{(\theta,0,k)}\big)
  & \colon l = 0
  \\
    F\big( \mathcal{X}_{l,m}^{(\theta,l,k)} (\mathcal{R}_{t}^{(\theta,l,k)}), Z^{(\theta,l,k)} \big)
     - F\big( \mathcal{X}_{l-1,m}^{(\theta,l,-k)}(\mathcal{R}_{t}^{(\theta,l,k)}), Z^{(\theta,l,k)} \big)
  & \colon l > 0     
  \end{cases}}
\end{split}
\end{equation}
is an independent family of random variables, 
and
\item\label{it:propMLP:item4} for all $n \in \N_0$, $t \in [0,T]$ it holds that
$\Omega \ni \omega \mapsto \mathcal{X}_{n,m}^\theta (t, \omega)
\in \R^d$, $\theta \in \Theta$,
are identically distributed random variables.
\end{enumerate}
\end{lemma}
\begin{proof}[Proof of \cref{propMLP}]
We first prove item~\eqref{it:propMLP:item1} by induction on $n \in \N_0$.
For the base case $n=0$ observe that the hypothesis that for all $\theta \in \Theta$, $t \in [0,T]$ it holds that 
${\mathcal{X}}_{0,m}^\theta(t) = \xi$ demonstrates that for all $\theta \in \Theta$ it holds that  
${\mathcal{X}}_{0,m}^\theta \colon [0,T] \times\Omega \rightarrow {\R}^d $ is a stochastic process with continuous sample paths.
This establishes item~\eqref{it:propMLP:item1} in the base case $n=0$.
For the induction step $\N_0 \ni (n-1) \rightarrow n \in \N$ let $n \in \N$ and assume that for every
$j \in \N_0 \cap [0,n)$, $\theta \in \Theta$ it holds
that ${\mathcal{X}}_{j,m}^\theta \colon [0,T] \times\Omega \rightarrow {\R}^d $ is a stochastic process with continuous sample paths.
Observe that the induction hypothesis, the hypothesis that $F\colon{\R}^d\times
S\rightarrow {\R}^d$ is $(\mathcal{B}({\R}^d) \otimes \mathcal{S}) / \mathcal{B}({\R}^d)$-measurable, the fact that for all $\theta \in \Theta$ it holds that $Z^\theta$ is $\mathcal{F}/\mathcal{S}$-measurable, the fact that 
for all $\theta \in \Theta$ it holds that 
$\mathcal{R}^\theta \colon
[0,T]\times\Omega \rightarrow [0,T]$ are stochastic processes, \cref{l2.14}, and \eqref{numsol}
prove that for all $\theta \in \Theta$ it holds that
${\mathcal{X}}_{n,m}^\theta \colon [0,T] \times \Omega \rightarrow {\R}^d $ is a stochastic process. This, the induction hypothesis, the fact that 
for all $\theta \in \Theta$ it holds that 
$\mathcal{R}^\theta \colon
[0,T]\times\Omega \rightarrow [0,T]$ are stochastic processes with continuous sample paths, \eqref{lipsch}, and \eqref{numsol} ensure that for all $\theta \in \Theta$ it holds that
${\mathcal{X}}_{n,m}^\theta \colon [0,T] \times \Omega \rightarrow {\R}^d $ is a stochastic process with continuous sample paths.
Induction thus establishes item~\eqref{it:propMLP:item1}.
Next we prove item~\eqref{it:propMLP:item2} by induction on $n \in \N_0$.
For the base case $n=0$ observe that the hypothesis that for all $\theta \in \Theta$, $t \in [0,T]$ it holds that 
${\mathcal{X}}_{0,m}^\theta(t) = \xi$ demonstrates that for all $\theta \in \Theta$ it holds that 
${\mathcal{X}}_{0,m}^\theta \colon [0,T] \times\Omega \rightarrow {\R}^d $ is
$\big( \mathcal{B}([0,T])\otimes\sigma_\Omega
		({(\mathfrak{r}^{(\theta,\vartheta)})}_{\vartheta \in \Theta},{(Z^{(\theta,\vartheta)})}_{\vartheta \in \Theta})\big)/\mathcal{B}(\R^d)$-measurable.
This implies item~\eqref{it:propMLP:item3} in the base case $n=0$.
For the induction step $\N_0 \ni (n-1) \rightarrow n \in \N$ let $n \in \N$ and assume that for every
$j \in \N_0 \cap [0,n)$, $\theta \in \Theta$ it holds
that ${\mathcal{X}}_{j,m}^\theta$ is $\big( \mathcal{B}([0,T])\otimes\sigma_\Omega
		({(\mathfrak{r}^{(\theta,\vartheta)})}_{\vartheta \in \Theta},{(Z^{(\theta,\vartheta)})}_{\vartheta \in \Theta})\big)/\mathcal{B}(\R^d)$-measurable.
Note that the induction hypothesis, the fact that $F\colon{\R}^d\times
S\rightarrow {\R}^d$ is $(\mathcal{B}({\R}^d) \otimes \mathcal{S}) / \mathcal{B}({\R}^d)$-measurable, the fact that for all $\theta \in \Theta$ it holds that $Z^\theta$ is $\mathcal{F}/\mathcal{S}$-measurable, \eqref{numsol}, and \cref{l2.14} prove that for all $\theta \in \Theta$, $t \in [0,T]$ it holds that
\begin{equation} \label{measurable}
\begin{split}
    \sigma_\Omega(\mathcal{X}^\theta_{n,m}(t)) 
    & \subseteq
    \sigma_\Omega\Big((Z^{(\theta,l,k)})_{k \in \lbrace 1,2, \dots,m^{n-l}\rbrace,l \in \N \cap [1,n)},
    (\mathcal{R}^{(\theta,l,k)}_t)_{ k \in \lbrace 1,2, \dots,m^{n-l} \rbrace, l \in \N \cap [1,n)},  \\
    &\ \ \ \ (\mathcal{R}^{(\theta,l,k, \vartheta)}_t)_{k \in \lbrace 1,2, \dots,m^{n-l} \rbrace, l \in \N \cap [1,n), \vartheta \in \Theta}, 
    (Z^{(\theta,l,k, \vartheta)})_{k \in \lbrace 1,2, \dots,m^{n-l} \rbrace,
    l \in \N \cap [1,n), \vartheta \in \Theta}, \\
    & \ \ \ \
    (\mathcal{R}^{(\theta,l,-k, \vartheta)}_t)_{k \in \lbrace 1,2, \dots,m^{n-l} \rbrace, l \in \N \cap [1,n), \vartheta \in \Theta},(Z^{(\theta,l,-k, \vartheta)})_{k \in \lbrace 1,2, \dots,m^{n-l} \rbrace,
    l \in \N \cap [1,n),
    \vartheta \in \Theta}, \\
    & \ \ \ \ (Z^{(\theta,0,k)})_{k \in \lbrace 1,2, \dots,m^{n} \rbrace}
    \Big) \\
    & \subseteq 
    \sigma_\Omega\big( (\mathfrak{r}^{(\theta, \vartheta)})_{\vartheta \in \Theta}, (Z^{(\theta, \vartheta)})_{\vartheta \in \Theta}
     \big).
\end{split}    
\end{equation}
Moreover, observe that item~\eqref{it:propMLP:item1} ensures that for all $\theta \in \Theta$ it holds that $\mathcal{X}^\theta_{n,m}$ is 
$\big( \mathcal{B}([0,T])\otimes\sigma_\Omega(\mathcal{X}^\theta_{n,m}) \big) / \mathcal{B}(\R^d)$-measurable. Combining this with
\eqref{measurable} demonstrates that for all $\theta \in \Theta$
it holds that $\mathcal{X}^\theta_{n,m}$ is
$\big( \mathcal{B}([0,T])\otimes\sigma_\Omega
({(\mathfrak{r}^{(\theta,\vartheta)})}_{\vartheta
\in \Theta}, \allowbreak {(Z^{(\theta,\vartheta)})}_{\vartheta
\in \Theta})\big)/\mathcal{B}(\R^d)$-measurable.
Induction thus establishes item~\eqref{it:propMLP:item2}.
Furthermore, observe that item~\eqref{it:propMLP:item2}, the hypothesis that
$(Z^\theta)_{\theta \in \Theta}$ are independent, 
the hypothesis that $(\mathfrak{r}^\theta)_{\theta \in \Theta}$ are independent, the hypothesis that $(Z^\theta)_{\theta \in \Theta}$
and $(\mathfrak{r}^\theta)_{\theta \in \Theta}$ are independent,
and \cref{l2.14} prove item~\eqref{it:propMLP:item3}.
Next we prove item~\eqref{it:propMLP:item4} by induction on $n \in \N_0$.
For the base case $n=0$ observe that the hypothesis that for all $\theta \in \Theta$, $t \in [0,T]$ it holds that 
${\mathcal{X}}_{0,m}^\theta(t) = \xi$ demonstrates that for all $t \in [0,T]$ it holds that
$\Omega \ni \omega \mapsto \mathcal{X}_{0,m}^\theta (t, \omega)
\in \R^d$, $\theta \in \Theta$, are identically distributed random variables.
This establishes item~\eqref{it:propMLP:item4} in the base case $n=0$.
For the induction step $\N_0 \ni (n-1) \rightarrow n \in \N$ let $n \in \N$ and assume that for every
$j \in \N_0 \cap [0,n)$, $t \in [0,T]$,  $\theta \in \Theta$ it
holds that $\Omega \ni \omega \mapsto \mathcal{X}_{j,m}^\theta (t, \omega)
\in \R^d$, $\theta \in \Theta$, are identically distributed random variables.
The induction hypothesis, item~\eqref{it:propMLP:item1}, the hypothesis that
$(Z^\theta)_{\theta \in \Theta}$ are i.i.d., 
the hypothesis that $(\mathfrak{r}^\theta)_{\theta \in \Theta}$ are i.i.d., the hypothesis that $(Z^\theta)_{\theta \in \Theta}$
and $(\mathfrak{r}^\theta)_{\theta \in \Theta}$ are independent,
item~\eqref{it:propMLP:item2}, \cref{l3.4}, and \cref{l3.5} (applied for every $ \vartheta \in \Theta \setminus \theta$, $\theta \in \Theta$,
$t \in [0,T]$, $l \in \N \cap [1,n)$, $k \in \N$ with $S \curvearrowleft [0,T]$, $E \curvearrowleft \R^d$, $U \curvearrowleft \big( F\big(\mathcal{X}^{(\theta,l,k)}_{l,m}(s), Z^{(\theta, l, k)}\big)-F\big(\mathcal{X}^{(\theta,l,-k)}_{l-1,m}(s), Z^{(\theta, l, k)}\big)\big)_{s \in [0,T]}$,
$V \curvearrowleft \big( F\big(\mathcal{X}^{(\vartheta,l,k)}_{l,m
}(s), Z^{(\vartheta, l,
k)}\big)-F\big(\mathcal{X}^{(\vartheta,l,-k)}_{l-1,m}(s), Z^{(\vartheta, l, k)}\big)\big)_{s \in [0,T]}$, $X \curvearrowleft \mathcal{R}^{(\theta,l,k)}_t$, $Y \curvearrowleft \mathcal{R}^{(\vartheta,l,k)}_t$ in the notation of \cref{l3.5}) assure that for all
$t \in [0,T]$, $l \in \N \cap [1,n)$, $k \in \N$ it holds that 
\begin{equation}
    \Big( F\big( \mathcal{X}_{l,m}^{(\theta,l,k)} (\mathcal{R}_{t}^{(\theta,l,k)}), Z^{(\theta,l,k)} \big)
     - F\big( \mathcal{X}_{l-1,m}^{(\theta,l,-k)}(\mathcal{R}_{t}^{(\theta,l,k)}), Z^{(\theta,l,k)} \big)\Big)_{\theta \in \Theta}
\end{equation}
are identically distributed random variables. Item (iii), \eqref{numsol}, the hypothesis that
$(Z^\theta)_{\theta \in \Theta}$ are i.i.d., and \cref{l3.4} therefore ensure that for all $t \in [0,T]$ it holds that
$\Omega \ni \omega \mapsto \mathcal{X}_{n,m}^\theta (t,\omega)
\in \R^d$, $\theta \in \Theta$, are identically distributed random variables.
Induction thus establishes item~\eqref{it:propMLP:item4}. The proof of \cref{propMLP} is thus completed.
\end{proof}

\begin{lemma} \label{l3.7}
Assume \cref{setting}, let $\theta \in \Theta$, $t \in [0, T]$, let 
$U_1 \colon [0,t]\times\Omega \rightarrow [0, \infty)$ and 
$U_2 \colon [0,t]\times\Omega \rightarrow \R^d$ be
stochastic processes with continuous sample paths, assume for all $i \in \lbrace1,2\rbrace$ that $U_i$ and $\mathfrak{r}^\theta$ are independent, and assume that
$\int_0^t \E[\| U_2(r) \|]\, dr <\infty$. Then it holds for all $i \in \lbrace1,2\rbrace$ that $\mathrm{Borel}_{[0,t]}(\lbrace r \in [0,t] \colon\E[\| U_2(r)\| ] =\infty\rbrace) =0$,
$\E[\| U_2(\mathcal{R}^\theta_t) \| ] < \infty $, and 
\begin{equation} \label{183}
t\, \E[ U_i(\mathcal{R}^\theta_t) ] = \int_0^t \E[ U_i(r)]\, dr.
\end{equation}
\end{lemma}

\begin{proof}[Proof of \cref{l3.7}]
Throughout this proof assume w.l.o.g.\ that $T>0$ and $t>0$.
Observe that the hypothesis that 
$\mathcal{R}^\theta_t =  \mathfrak{r}^\theta t$ implies that $\mathcal{R}^\theta_t$
is $\mathcal{U}_{[0,t]}$-distributed. Combining this with the 
fact that $U_1$ is a stochastic process with continuous sample paths, the fact that $U_1$ and
$\mathcal{R}^\theta_t$ are independent, and \cref{l2.15}
assures that
\begin{equation} \label{184}
\begin{split}
    t\,  \E[ U_1(\mathcal{R}^\theta_t) ]
    & =
    t \int_{[0,t]} \E[ U_1(r)] (\mathcal{R}^\theta_t(\P)_
    {\mathcal{B}([0,t])})(dr)\\
    &= t \int_{[0,t]} \E[ U_1(r)] (\mathcal{U}_{[0,t]})(dr)
    = \frac{t}{t} \int_0^t \E[ U_1(r)]\, dr 
    = \int_0^t \E[ U_1(r)]\, dr.
\end{split}
\end{equation}
In addition, note that the fact that $\mathcal{R}^\theta_t$
is $\mathcal{U}_{[0,t]}$-distributed, the fact that 
$U_2$ is a stochastic process with continuous sample paths, the fact that $U_2$ and $\mathcal{R}^\theta_t$ are independent, the hypothesis that
$\int_0^t \E[\| U_2(r) \|]\, dr <\infty$, and \cref{l2.16new}
ensure that $\mathrm{Borel}_{[0,t]}(\lbrace r \in [0,t] \colon \E[\| U_2(r) \|] = \infty  \rbrace)=0$, 
$\E[\|U_2(\mathcal{R}^\theta_t)\|]< \infty$, and
\begin{equation}
    \begin{split}
    t\,  \E[ U_2(\mathcal{R}^\theta_t) ] 
    & =
    t \int_{[0,t]} \E[ U_2(r)] (\mathcal{R}^\theta_t(\P)_
    {\mathcal{B}([0,t])})(dr)\\
    &= t \int_{[0,t]} \E[ U_2(r)] (\mathcal{U}_{[0,t]})(dr)
    = \frac{t}{t} \int_0^t \E[ U_2(r)]\, dr 
    = \int_0^t \E[ U_2(r)]\, dr.
    \end{split}
\end{equation}
Combining this with \eqref{184} establishes \eqref{183}.
The proof of \cref{l3.7} is thus completed.
\end{proof}

\begin{lemma}[Expectations of MLP approximations]
\label{expMLP}     
Assume \cref{setting} and assume that
$\E [ \| F(\xi, Z^0) \| ] < \infty $. Then 
\begin{enumerate}[(i)]
\item\label{it:expMLP:item1} for all $n \in \N_0$, $m \in \N$, $t \in [0,T]$, $s \in [0,t]$ it holds that
\begin{equation}
\begin{split}
&\E\big[\| \mathcal{X}_{n,m}^0 (s) \| \big] + t\, \E\big[ \|
 \mathcal{X}_{n,m}^0 (\mathcal{R}_t^0) \| \big] + 
t\, \E  \big[ \|F(  \mathcal{X}_{n,m}^0 (\mathcal{R}_t^0), Z^0) \| 
 \big] \\
 & = \E \big[\| \mathcal{X}_{n,m}^0 (s) \| \big]
 + \int_0^t \E \big[ \| \mathcal{X}_{n,m}^0 (r) \| \big]\, dr  +\int_0^t \E  \big[
\|F(  \mathcal{X}_{n,m}^0 (r), Z^0) \|
 \big]\, dr <\infty
 \end{split}
\end{equation}
and
\item\label{it:expMLP:item2} for all $n, m \in \N$, $t \in [0,T]$ it holds that 
$\E[ \mathcal{X}_{n,m}^0 (t)] = \xi + \int_0^t \E [ F(  \mathcal{X}_{n-1,m}^0 (r), Z^0 )]\, dr $.
\end{enumerate}
\end{lemma}

\begin{proof}[Proof of \cref{expMLP}]
Throughout this proof let $m \in \N$. Observe that \cref{l3.7}, items~\eqref{it:propMLP:item1} and \eqref{it:propMLP:item2} in \cref{propMLP}, and the fact that for all $n \in \N$ it holds that $\mathcal{X}_{n,m}^0$, $Z^0$, and $\mathfrak{r}^0$ are independent demonstrate that for all $n \in \N_0$, $ t \in [0,T] $ it holds that
\begin{equation} \label{188}
t\, \E\big[ \|  \mathcal{X}_{n,m}^0 (\mathcal{R}_t^0) \| \big]  + 
t\, \E  \big[ \|F(  \mathcal{X}_{n,m}^0 (\mathcal{R}_t^0), Z^0) \| 
 \big] 
 = \int_0^t \E \big[ \| \mathcal{X}_{n,m}^0 (r)\| \big]\, dr
 +   \int_0^t \E  \big[
 \|F(  \mathcal{X}_{n,m}^0 (r), Z^0) \|
 \big]\, dr .
\end{equation}
 Next we claim that for all $n \in \N_0$, $t \in [0,T]$, $s \in [0,t]$
it holds that 
\begin{equation} \label{induction}
  \E \big[ \| \mathcal{X}_{n,m}^0 (s) \| \big]
 + \int_0^t \E \big[\| \mathcal{X}_{n,m}^0 (r) \| \big]\, dr 
  +   \int_0^t \E  \big[
 \|F(  \mathcal{X}_{n,m}^0 (r), Z^0) \|
 \big]\, dr <\infty.
\end{equation}
We now prove \eqref{induction} by induction on $n \in \N_0$. For the base case $n=0$ observe that
 the hypothesis that for all $t \in [0,T]$ it holds that $\mathcal{X}_{0,m}^0 (t)
 = \xi $ and the hypothesis that 
$\E [ \| F(\xi, Z^0) \| ] < \infty $  imply that for all 
$t \in [0,T]$, $s \in [0,t]$ it holds that
\begin{equation}
\E\big[ \| \mathcal{X}_{0,m}^0 (s) \| \big] +
 \int_0^t \E \big[ \| \mathcal{X}_{0,m}^0 (r) \| \big]\, dr 
  +   \int_0^t \E  \big[
 \|F(  \mathcal{X}_{0,m}^0 (r), Z^0) \|
 \big]\, dr
 \leq \left\| \xi \right\| + T \| \xi \| +
T\, \E [ \| F(\xi, Z^0) \| ] < \infty.
\end{equation}
This establishes \eqref{induction} in the base case $n=0$. 
 For the induction step $\N_0\ni (n-1) \rightarrow n \in \N$ 
let $n \in \N$ and assume that for all $j \in \N_0 \cap [0,n)$, 
$t \in [0,T]$, $s \in [0,t]$ it holds that
\begin{equation} \label{191}
 \E \big[ \| \mathcal{X}_{j,m}^0 (s) \| \big]
 + \int_0^t \E \big[ \| \mathcal{X}_{j,m}^0 (r) \| \big]\, dr 
  +   \int_0^t \E  \big[
 \|F(  \mathcal{X}_{j,m}^0 (r), Z^0) \|
 \big]\, dr <\infty.
\end{equation}
Note that \eqref{numsol} and the triangle inequality ensure that 
for all $t \in [0,T]$, $s \in [0,t]$ it holds that
\begin{equation} \label{192}
\begin{split}
\E\big[ \| \mathcal{X}_{n,m}^0 (s) \| \big]
 & \leq     \sum_{l=1}^{n-1}
\frac{s}{m^{n-l}}  \Bigg[  \sum_{k=1}^{m^{n-l}} 
 \bigg( \E \big[ \| F ( \mathcal{X}_{l,m}^{(0,l,k)} (\mathcal{R}_{s}^{(0,l,k)}), Z^{(0,l,k)} ) \| \big]  + \E \big[\| F( \mathcal{X}_{l-1,m}^{(0,l,-k)}(\mathcal{R}_{s}^{(0,l,k)}), Z^{(0,l,k)} )\| \big] \bigg)  \Bigg] \\
 & + \frac{s}{m^n}  \Bigg[ \sum_{k=1}^{m^n}
 \E \big[ \| F(\xi,Z^{(0,0,k)} ) \| \big] \Bigg] + \| \xi \|.
\end{split}
\raisetag{65pt}
\end{equation}
 Furthermore, observe that the hypothesis that $\E \big[ \| F\big(\xi,Z^0 \big) \| \big]  < \infty$ and the hypothesis that $(Z^\theta)_{\theta \in \Theta}$ are identically distributed random variables assure that for all $k \in \mathbb{Z}$ it holds that
 \begin{equation}\label{193}
 \E \big[ \| F(\xi,Z^{(0,0,k)} ) \| \big] =
 \E \big[ \| F(\xi,Z^0) \| \big]  < \infty.
 \end{equation}
Moreover, observe that \cref{l3.7}, the hypothesis that $ {(Z^{\theta})}_{\theta \in \Theta} $ are i.i.d., the hypothesis that ${(\mathfrak{r}^\theta)}_{\theta \in \Theta}$ are independent, the hypothesis that  $ {(Z^{\theta})}_{\theta \in \Theta} $ and ${(\mathfrak{r}^\theta)}_{\theta \in \Theta}$ are independent, and  items~\eqref{it:propMLP:item1}, \eqref{it:propMLP:item2}, and \eqref{it:propMLP:item4} in \cref{propMLP} demonstrate that for all $i,j,k \in \mathbb{Z}$, $l \in \N_0$, $t \in [0,T]$, $s \in [0,t]$
 it holds that
\begin{equation}
 s\,  \E \big[ \| F ( \mathcal{X}_{l,m}^{(0,j,i)} (\mathcal{R}_{s}^{(0,j,k)}), Z^{(0,j,k)} ) \| \big] 
 = \int_0^s  \E \big[ \| F ( \mathcal{X}_{l,m}^{(0,j,i)} (r), Z^{(0,j,k)} ) \| \big]\,  dr = \int_0^s  \E \big[ \| F ( \mathcal{X}_{l,m}^0 (r), Z^0) \| \big]\, dr.
\end{equation}
Combining this, \eqref{191}, \eqref{192}, and \eqref{193} establishes that for all $t \in [0,T]$, $s \in [0,t]$ it holds that
\begin{equation}  \label{195}
\begin{split}
\E\big[ \| \mathcal{X}_{n,m}^0 (s) \| \big] 
& \leq    
\Bigg(  \sum_{l=1}^{n-1}
\frac{1}{m^{n-l}}  \Bigg[  \sum_{k=1}^{m^{n-l}} 
 \bigg(\int_0^s  \E \big[ \| F ( \mathcal{X}_{l,m}^0 (r), Z^0) \| \big]\, dr  +  \int_0^s  \E \big[ \| F ( \mathcal{X}_{l-1,m}^0 (r), Z^0) \| \big]\, dr \bigg)\Bigg] \Bigg) 
  \\
  & 
  \quad + \frac{s}{m^n}  \left[ \sum_{k=1}^{m^n}
  \E \big[ \| F(\xi,Z^0) \| \big]  \right] + \| \xi \| \\
  & = \Bigg[  \sum_{l=1}^{n-1} \bigg(
 \int_0^s  \E \big[ \| F ( \mathcal{X}_{l,m}^0 (r), Z^0) \| \big]\,  dr 
  +  \int_0^s  \E \big[ \| F ( \mathcal{X}_{l-1,m}^0 (r), Z^0) \| \big]\, dr \bigg) \Bigg] 
  +    
 s\,  \E \big[ \| F(\xi,Z^0) \| \big] + \| \xi\| \\
 & \leq  2 \Bigg[  \sum_{l=1}^{n-1} 
 \int_0^t  \E \big[ \| F ( \mathcal{X}_{l,m}^0 (r), Z^0) \| \big]\,  dr  \Bigg]  + 2T\, \E \big[ \| F(\xi,Z^0) \| \big] 
 + \| \xi \| <\infty.
\end{split}
\end{equation}
Hence, we obtain that for all $t \in [0,T]$ it holds that 
\begin{equation} \label{196}
\begin{split}
&\int_0^t \E\big[ \| \mathcal{X}_{n,m}^0 (r) \| \big]\, dr
\leq t  \sup_{s \in [0,t]} \E\big[ \| \mathcal{X}_{n,m}^0 (s) \| \big]\\
& \leq t \left( 2 \Bigg[  \sum_{l=1}^{n-1} 
 \int_0^t  \E \big[ \| F ( \mathcal{X}_{l,m}^0 (r), Z^0) \| \big]\,  dr  \Bigg] + 2T\, \E \big[ \| F(\xi,Z^0) \| \big]  + \left\| \xi \right\| \right) <\infty.
\end{split}
\end{equation}
Next note that \eqref{lipsch} and the triangle inequality imply that for all $x \in {\R}^d$, $s \in S$ it holds that
\begin{equation} 
\| F(x,s) \|  \leq \| F(\xi,s) \| + \| F(x,s) -F(\xi,s) \| \leq 
 \| F(\xi,s) \| + L \| x-\xi \|.
\end{equation}
This, the triangle inequality, \eqref{196}, and the hypothesis that $\E [ \| F(\xi, Z^0) \| ] < \infty $ assure that
for all $t \in [0,T]$ it holds that
\begin{equation}
\begin{split}
\int_0^t  \E \big[ \| F ( \mathcal{X}_{n,m}^0 (r), Z^0) \| \big]\,  dr 
& \leq \int_0^t \E \big[ \| F ( \xi, Z^0) \| \big]\,  dr 
+ L\!  \int_0^t \E \big[ \|  \mathcal{X}_{n,m}^0 (r) - \xi \| \big]\,  dr \\
& \leq T\,  \E \big[ \| F ( \xi, Z^0) \| \big]
+ L\! \int_0^t \E \big[\|  \mathcal{X}_{n,m}^0 (r) \| \big]\,  dr + L\! \int_0^t \E \big[ \|  \xi \| \big]\,  dr \\
& \leq  T\,  \E \big[ \| F ( \xi, Z^0) \| \big]
+ L\!  \int_0^t \E \big[ \|  \mathcal{X}_{n,m}^0 (r) \| \big]\,  dr
+ L T \left\|  \xi \right\| <\infty .
\end{split}
\end{equation}
This, \eqref{195}, and \eqref{196} establish that for all $t \in [0,T]$, $s \in [0,t]$ it holds that
\begin{equation}
 \E \big[ \| \mathcal{X}_{n,m}^0 (s) \| \big]
 + \int_0^t \E \big[\| \mathcal{X}_{n,m}^0 (r) \| \big]\,
 dr 
  +   \int_0^t \E  \big[
 \| F(  \mathcal{X}_{n,m}^0 (r), Z^0) \|
 \big]\, dr <\infty.
\end{equation}
Induction thus proves \eqref{induction}.
Combining \eqref{188} and \eqref{induction} hence establishes item~\eqref{it:expMLP:item1}.
Next observe that  \eqref{numsol}, \eqref{induction}, items~\eqref{it:propMLP:item1}, \eqref{it:propMLP:item2}, and \eqref{it:propMLP:item4} in \cref{propMLP}, the hypothesis that $ {(Z^{\theta})}_{\theta \in \Theta} $ are i.i.d., the hypothesis that ${(\mathfrak{r}^\theta)}_{\theta \in \Theta}$ are i.i.d., the hypothesis that  $ {(Z^{\theta})}_{\theta \in \Theta} $ and ${(\mathfrak{r}^\theta)}_{\theta \in \Theta}$ are independent,
and \cref{l3.5} ensure that for all $n \in \N$,
$t \in [0,T]$ it holds that
\begin{equation} \label{mean}
\begin{split}
\E \big[ \mathcal{X}_{n,m}^0 (t)\big] 
&= \sum_{l=1}^{n-1}
\frac{t}{m^{n-l}}  \Bigg[  \sum_{k=1}^{m^{n-l}} \!\bigg( \E \big[ F ( \mathcal{X}_{l,m}^{(0,l,k)} (\mathcal{R}_{t}^{(0,l,k)}), Z^{(0,l,k)} )  \big] - \E \big[  F( \mathcal{X}_{l-1,m}^{(0,l,-k)}(\mathcal{R}_{t}^{(0,l,k)}), Z^{(0,l,k)} ) \big] \bigg) \Bigg] 
\\
& \quad + \frac{t}{m^n}  \left[ \sum_{k=1}^{m^n}
 \E \big[  F(\xi,Z^{(0,0,k)}) \big] \right] + \xi \\
& =  t  \left[ \sum_{l=1}^{n-1} \bigg(  \E \big[ F ( \mathcal{X}_{l,m}^0 (\mathcal{R}_t^0), Z^0 ) \big] 
 - \E \big[  F( \mathcal{X}_{l-1,m}^0
(\mathcal{R}_t^0), Z^0) \big] \bigg)\right] + t\, \E \big[  F(\xi,Z^0 )\big] + \xi \\
&= \xi + t\, \E \big[  F( \mathcal{X}_{n-1,m}^0
(\mathcal{R}_t^0), Z^0 ) \big].
\end{split}
\raisetag{96pt}
\end{equation}
\cref{l3.7}, items~\eqref{it:propMLP:item1} and \eqref{it:propMLP:item4} in \cref{propMLP}, the fact that for all $n \in \N_0$ it holds that $\mathcal{X}_{n,m}^0$, $Z^0$, and $\mathfrak{r}^0$ are independent, and \eqref{induction} hence imply that for all $n \in \N$, $ t \in [0,T] $ it holds that
\begin{equation}
\E \big[ \mathcal{X}_{n,m}^0 (t) \big]  =
  \xi + \int_0^t \E \big[F(  \mathcal{X}_{n-1,m}^0 (r), Z^0)\big]\, dr .
\end{equation}
This establishes item~\eqref{it:expMLP:item2}. 
The proof of \cref{expMLP} is thus completed.
\end{proof}

\subsection{Error estimates for MLP approximations }\label{subsectionerrorestimate}

\begin{lemma}\label{variances}
Let $n \in {\N}$, let $(\Omega,\mathcal{F},{\P})$ be a probability
 space, let $X_1, X_2, \dots , X_n \colon \\ \Omega\rightarrow
 {\R}$ be independent random variables, and assume for all $i \in
 \left\lbrace 1,2, \dots , n\right\rbrace $ that ${\E}[|X_i |] < \infty$. Then it holds that 
\begin{equation}
\mathrm{Var} \Bigg( \sum_{i=1}^n X_i \Bigg) 
 = {\E} \Big[
\big| {\E}  \big[ \textstyle\sum_{i=1}^{n} X_i \big] - \textstyle\sum_{i=1}^n X_i \big| ^2  \Big] =\sum_{i=1}^n {\E}  \big[ | {\E}[X_i ] - X_i |^2 \big] =  \sum_{i=1}^n \mathrm{Var}(X_i).
\end{equation}
\end{lemma}


\begin{lemma} \label{l2.9}
Let $(\Omega, \mathcal{F}, \mu)$ be a measure space and let $f \colon \Omega\rightarrow [0,\infty]$ be $\mathcal{F} / \mathcal{B}([0,\infty])$-measurable. Then it holds that
\begin{equation}
{\left[ \int_{\Omega} f(\omega)\,  \mu(d\omega)  \right] }^2
\leq \mu(\Omega) \int_{\Omega} {|f(\omega)|}^2 \, \mu(d\omega).
\end{equation} 
\end{lemma}


\begin{lemma}\label{l3.14}
Let  $T \in [0,\infty)$, $k \in \N$, and let $U \colon [0,T]\rightarrow
[0,\infty]$ be $\mathcal{B}([0,T])/\mathcal{B}([0,\infty])$-measurable.
Then it holds that
\begin{equation}
\int_0^T \frac{(T-t)^{k-1}}{(k-1)!} \int_0^t U(r)\, dr\, dt = 
\int_0^T \frac{(T-t)^k}{k!} U(t)\, dt.
\end{equation}
\end{lemma}
\begin{proof}[Proof of \cref{l3.14}]
Observe that Tonelli's theorem assures that
\begin{equation}
\begin{split}
\int_0^T \frac{(T-t)^{k-1}}{(k-1)!} \int_0^t U(r)\, dr\, dt
& = \int_0^T \int_0^T  \frac{(T-t)^{k-1}}{(k-1)!} U(r)\   
\mathbbm{1}_{ \lbrace  (\mathfrak{t},\mathfrak{r}) \in [0,T]^2 \colon \mathfrak{r} \leq \mathfrak{t} \rbrace} (t,r) \, dr\, dt \\
& = \int_0^T \int_0^T  \frac{(T-t)^{k-1}}{(k-1)!} U(r)\   
\mathbbm{1}_{ \lbrace  (\mathfrak{t},\mathfrak{r}) \in [0,T]^2 \colon \mathfrak{r} \leq \mathfrak{t} \rbrace} (t,r) \, dt\, dr 
\\
& = \int_0^T \int_r^T  \frac{(T-t)^{k-1}}{(k-1)!} \, dt \  U(r) \, dr
= \int_0^T \frac{(T-r)^k}{k!}  U(r)\, dr .
\end{split}
\end{equation}
The proof of \cref{l3.14} is thus completed.
\end{proof}

\begin{proposition}\label{thmerrorestimate}
Assume \cref{setting}.
Then it holds for all $n \in {\N_0}$, $m \in \N$ that 
\begin{equation}
\big( {\E} \big[\| X(T) - {\mathcal{X}}_{n,m}^0(T)\|^2\big] 
\big)^{1/2} \leq \frac{ T
 \big( {\E}  \big[ \|F(\xi, Z^0)\|^2 \big] \big) ^{1/2}(1+2LT)^n e^{(LT + m/2)}}{m^{n/2}} .
\end{equation}
\end{proposition}

\begin{proof}[Proof of \cref{thmerrorestimate}]
Throughout this proof assume w.l.o.g.\ that 
$T>0$ and ${\E}[ \|F(\xi, Z^0)\|^2]$  $< \infty$, let 
$C \in {[0,\infty) }$ satisfy that 
\begin{equation} \label{C}
C =  T
 \big( {\E}  \big[ \|F(\xi, Z^0) \|^2 \big] \big) ^{1/2}  e^{LT},
\end{equation}
let $\zeta_i \in \R$, $i \in \lbrace 1,2, \dots , d \rbrace$, satisfy that $(\zeta_1, \zeta_2, \dots , \zeta_d) = \xi$,
 let ${\chi}_{i,n,m}^\theta \colon [0,T]\times \Omega \rightarrow \R$, $n,m \in \N_0$, $\theta \in \Theta$, $i \in \lbrace 1,2, \dots , d \rbrace$,
 satisfy for all $n \in \N_0$, $m \in \N$, $\theta \in \Theta$ that 
 $({\chi}_{1,n,m}^\theta,{\chi}_{2,n,m}^\theta, \dots,{\chi}_{d,n,m}^\theta) = {\mathcal{X}}_{n,m}^\theta$,
 let $f_i \colon \R^d \times S \rightarrow \R$, $i \in \lbrace 1,2, \dots , d \rbrace$, satisfy for all $x \in \R^d$, $s \in S$ that
 \begin{equation}\label{leeffe}
     (f_1(x,s), f_2(x,s) \dots, f_d(x,s)) = F(x,s),
 \end{equation}
 and let $m \in \N$.
Observe that \eqref{numsol} establishes that
for all $n \in \N_0$, $\theta \in \Theta$, $t \in [0,T]$, $i \in \lbrace 1,2, \dots , d \rbrace$  it holds that
\begin{equation}
 \begin{split}\label{settingoncomponents}
{\mathcal{\chi}}_{i,n,m}^\theta (t) & = \Bigg[
 \sum_{l=1}^{n-1}
\frac{t}{m^{n-l}}  \sum_{k=1}^{m^{n-l}} \bigg( f_i\big( \mathcal{X}_{l,m}^{(\theta,l,k)} (\mathcal{R}_{t}^{(\theta,l,k)}), Z^{(\theta,l,k)} \big) - f_i\big( \mathcal{X}_{l-1,m}^{(\theta,l,-k)}(\mathcal{R}_{t}^{(\theta,l,k)}), Z^{(\theta,l,k)} \big)\bigg) \Bigg] \\
&  
^+\left[  \frac{t \mathbbm{1}_{\N}(n)}{m^n}  \sum_{k=1}^{m^n} f_i\big(\xi,Z^{(\theta,0,k)}\big)\right]  + \zeta_i.
 \end{split}
\end{equation}
Furthermore, note that \eqref{C} assures that
\begin{equation} \label{Csquared}
C^2 = T^2  {\E}  \big[ \| F(\xi, Z^0) \|^2 \big]  e^{2LT}
 \geq 
T^2 \big( {\E}  \big[ \|F(\xi, Z^0) \| ^2 \big] \big) .
\end{equation}
In addition, note that the hypothesis that $\E[ \|F(\xi, Z^0) \|^2] < \infty$ and Jensen's inequality assure that 
\begin{equation}\label{necperdopo}
\big(\E[\| F(\xi,Z^0) \|]\big)^2 \leq \E [ \|F(\xi, Z^0) \|^2 ] < \infty.
\end{equation}
This, \cref{expMLP},
 \eqref{lipsch}, and  \eqref{sol} 
 demonstrate that for all  $t \in [0,T]$, $n \in \N$ it holds that
\begin{equation}
\begin{split}
 \left\| {\E} [{\mathcal{X}}_{n,m}^0(t)] - X(t) \right\| 
 & \leq  \left\|  \xi + \int_0^t \E \big[ 
F(  \mathcal{X}_{n-1,m}^0 (r), Z^0) \big]\,  dr -
 \xi - \int_0^t {{\E}\big[ F (  X(r), Z^0 )  \big]}\, dr
 \right\|\\
  &= \left\|   \int_0^t \Big( \E\big[F(  \mathcal{X}_{n-1,m}^0 (r), Z^0)\big]-\E\big[F (  X(r), Z^0 )\big] \Big) \,   dr    \right\| \\
 &= \left\|   \int_0^t  \E\big[F(  \mathcal{X}_{n-1,m}^0 (r), Z^0)-F (  X(r), Z^0 )\big]\,   dr    \right\| \\
 & \leq  \int_0^t \E\big[\| F(  \mathcal{X}_{n-1,m}^0 (r), Z^0)-F(  X(r), Z^0)\|\big]   dr    \\
 &  \leq  \int_0^t \E \big[ L \|  \mathcal{X}_{n-1,m}^0 (r) - X(r) \| \big]\,   dr
  = L \int_0^t \E \big[  \|  \mathcal{X}_{n-1,m}^0 (r) - X(r) \|  \big]\, dr .
 \end{split}
\end{equation}
This, \cref{l2.9}, and Jensen's inequality imply that 
for all $t \in [0,T]$, $n \in \N$ it holds that 
\begin{equation} \label{eq0}
\begin{split}
\left\| {\E} [{\mathcal{X}}_{n,m}^0(t)] - X(t) \right\|^2 
& \leq L^2 \left( \int_0^t \E \big[  \|  \mathcal{X}_{n-1,m}^0 (r) - X(r) \|    \big]\,dr \right)^{\!\!2} 
\leq L^2 t  \int_0^t \big( \E \big[  \|  \mathcal{X}_{n-1,m}^0 (r) - X(r) \| \big]\! \big)^2 \, dr \\
&\leq L^2 T  \int_0^t \big( \E \big[ \|  \mathcal{X}_{n-1,m}^0 (r) - X(r) \|    \big] \big)^2 \,dr 
\leq L^2 T  \int_0^t \E \big[  \|  \mathcal{X}_{n-1,m}^0 (r) - X(r) \| ^2   \big]\, dr.
\end{split}
\end{equation}
In addition, observe that \cref{variances}, \eqref{necperdopo},  item~\eqref{it:expMLP:item1} in \cref{expMLP}, item~\eqref{it:propMLP:item3} in \cref{propMLP}, and \eqref{settingoncomponents} imply that for all $t \in [0,T]$, $n \in \N$ it holds that
\begin{equation} \label{eq1}
 \begin{split}
 & \E\big[ \|{\mathcal{X}}_{n,m}^0 (t) - \E[{\mathcal{X}}_{n,m}^0 (t)] \|^2 \big] 
 = \E \bigg[ \sum_{i=1}^d \big|{\chi}_{i,n,m}^0(t) - \E [{\chi}_{i,n,m}^0(t)]\big|^2 \bigg] 
 = \sum_{i=1}^d \mathrm{Var} 
\big({\chi}_{i,n,m}^0(t) \big)\\
&
 = \sum_{i=1}^d  \Bigg[  \sum_{l=1}^{n-1} 
  \sum_{k=1}^{m^{n-l}}  \mathrm{Var} \bigg( \frac{t}{m^{n-l}}\Big[  f_i \big( \mathcal{X}_{l,m}^{(0,l,k)} (\mathcal{R}_{t}^{(0,l,k)}), Z^{(0,l,k)} \big) - f_i  \big( \mathcal{X}_{l-1,m}^{(0,l,-k)}(\mathcal{R}_{t}^{(0,l,k)}), Z^{(0,l,k)} \big) \Big]  \bigg)
  \\
  & \quad + \mathrm{Var} \bigg(\frac{t}{m^n} \sum_{k=1}^{m^n}  f_i \big(\xi,Z^{(0,0,k)}\big)\bigg)  \Bigg] .
\end{split}
\end{equation}
Moreover, note that the hypothesis that $(Z^\theta)_{\theta \in \Theta}$ are i.i.d.\ and the fact that for all $Y \in \mathcal{L}^1(\P ;\!\R)$ it holds that $\mathrm{Var}(Y) \leq
\E[|Y|^2]$ ensure that for all $i \in \lbrace1, 2, \dots , d \rbrace$, $t \in [0,T]$, $n \in \N$ it holds that
\begin{equation}
 \mathrm{Var} \bigg(\frac{t}{m^n} \sum_{k=1}^{m^n}  f_i\big(\xi,Z^{(0,0,k)}\big)\bigg) 
 =   \frac{t^2 }{m^{2n}}  \mathrm{Var}\bigg( \sum_{k=1}^{m^n}  f_i\big(\xi,Z^{(0,0,k)} \big)\bigg)
 = \frac{t^2 m^n }{m^{2n}}  \mathrm{Var} \big(f_i \big(\xi,Z^0 \big)\big) 
 \leq \frac{t^2}{m^n}
 \E [| f_i \big(\xi,Z^0 \big)|^2].
\end{equation}
This and \eqref{leeffe} imply that for all $t \in [0,T]$, $n \in \N$ it holds that
\begin{equation} \label{eq2}
\begin{split}
 \sum_{i=1}^d \mathrm{Var} \bigg(\frac{t}{m^n} \sum_{k=1}^{m^n}  f_i\big(\xi,Z^{(0,0,k)}\big)\bigg)
 & \leq
 \sum_{i=1}^d \frac{t^2}{m^n} \big(
 \E [| f_i (\xi,Z^0 ) | ^2] \big) 
 = 
  \frac{t^2}{m^n} \Big(
 \E \Big[ \sum_{i=1}^d| f_i (\xi,Z^0 ) | ^2 \Big] \Big) 
 \\
 & 
 =
 \frac{t^2}{m^n} \big(
 \E \big[ \| F\big(\xi,Z^0 \big) \| ^2 \big] \big) 
 \leq \frac{T^2}{m^n}
 \E \big[ \| F \big(\xi,Z^0 \big) \| ^2 \big].
\end{split}
\raisetag{2cm}
\end{equation}
In addition, observe that items~\eqref{it:propMLP:item1}, \eqref{it:propMLP:item2}, and \eqref{it:propMLP:item4} in \cref{propMLP}, the hypothesis that $(Z^\theta)_{\theta \in \Theta}$ are i.i.d., the hypothesis that ${(\mathfrak{r}^\theta)}_{\theta \in \Theta}$ are i.i.d., the hypothesis that $(Z^\theta)_{\theta \in \Theta}$ and 
${(\mathfrak{r}^\theta)}_{\theta \in \Theta}$ are independent,
 the fact that for all $Y \in \mathcal{L}^1(\P;\!\R)$ it holds that $\mathrm{Var}(Y) \leq \E[|Y|^2]$, and \cref{l3.5} imply that for all $i \in \lbrace1, 2, \dots , d \rbrace$, $t \in [0,T]$, $n \in \N$, $l \in \N\cap [1,n)$ it holds that
\begin{equation}
\begin{split} 
 & \sum_{k=1}^{m^{n-l}}  \mathrm{Var} \Big( \frac{t}{m^{n-l}}\Big[ f_i \big( \mathcal{X}_{l,m}^{(0,l,k)} (\mathcal{R}_{t}^{(0,l,k)}), Z^{(0,l,k)} \big) - f_i \big( \mathcal{X}_{l-1,m}^{(0,l,-k)}(\mathcal{R}_{t}^{(0,l,k)}), Z^{(0,l,k)} \big) \Big]  \Big) \\
 & = m^{n-l} \mathrm{Var}  \Big( \frac{t}{m^{n-l}}\Big[  f_i\big( \mathcal{X}_{l,m}^0 (\mathcal{R}_{t}^0), Z^0 \big)
 - f_i\big( \mathcal{X}_{l-1,m}^1(\mathcal{R}_{t}^0), Z^0 \big) \Big]  \Big) \\
 & =   \frac{m^{n-l}t^2}{m^{2(n-l)}} \mathrm{Var} \Big(  f_i\big( \mathcal{X}_{l,m}^0 (\mathcal{R}_{t}^0), Z^0 \big)  - f_i\big( \mathcal{X}_{l-1,m}^1(\mathcal{R}_{t}^0), Z^0 \big)   \Big)\\
 &\leq \frac{t^2}{m^{n-l}} \E \big[ \big| f_i\big( \mathcal{X}_{l,m}^0 (\mathcal{R}_{t}^0), Z^0 \big)  - f_i\big( \mathcal{X}_{l-1,m}^1(\mathcal{R}_{t}^0), Z^0 \big) \big| ^2  \big].
\end{split}
\end{equation}
This and \eqref{leeffe} ensure that for all $t \in [0,T]$, $n \in \N$, $l \in \N\cap [1,n)$ it holds that
\begin{equation} \label{eq3}
\begin{split} 
 & \sum_{i=1}^d \sum_{k=1}^{m^{n-l}}  \mathrm{Var} \bigg( \frac{t}{m^{n-l}}\Big[  f_i\big( \mathcal{X}_{l,m}^{(0,l,k)} (\mathcal{R}_{t}^{(0,l,k)}), Z^{(0,l,k)} \big)
  - f_i\big( \mathcal{X}_{l-1,m}^{(0,l,-k)}(\mathcal{R}_{t}^{(0,l,k)}), Z^{(0,l,k)} \big) \Big]  \bigg) \\
&\leq  \sum_{i=1}^d \frac{t^2}{m^{n-l}} \E \big[ | f_i\big( \mathcal{X}_{l,m}^0 (\mathcal{R}_{t}^0), Z^0 \big)  - f_i\big( \mathcal{X}_{l-1,m}^1(\mathcal{R}_{t}^0), Z^0 \big)  | ^2  \big]\\
&= \frac{t^2}{m^{n-l}} \E \bigg[  \sum_{i=1}^d | f_i\big( \mathcal{X}_{l,m}^0 (\mathcal{R}_{t}^0), Z^0 \big)  - f_i\big( \mathcal{X}_{l-1,m}^1(\mathcal{R}_{t}^0), Z^0 \big) | ^2  \bigg]\\
&= \frac{t^2}{m^{n-l}} \E \big[  \| F \big( \mathcal{X}_{l,m}^0 (\mathcal{R}_{t}^0), Z^0 \big)  - F\big( \mathcal{X}_{l-1,m}^1(\mathcal{R}_{t}^0), Z^0 \big) \| ^2  \big].
\end{split}
\end{equation}
Combining this, \eqref{eq1}, and \eqref{eq2}  ensures that for all $t \in [0,T]$, $n \in \N$ it holds that
\begin{equation} \label{eq4}
\begin{split}
\E\big[ \|{\mathcal{X}}_{n,m}^0 (t) - \E[{\mathcal{X}}_{n,m}^0 (t)] \|^2 \big] 
&= \sum_{i=1}^d  \Bigg[     
 \sum_{l=1}^{n-1} 
  \sum_{k=1}^{m^{n-l}}  \mathrm{Var} \bigg( \frac{t}{m^{n-l}}\Big[  f_i\big( \mathcal{X}_{l,m}^{(0,l,k)} (\mathcal{R}_{t}^{(0,l,k)}), Z^{(0,l,k)} \big) \\
&\ \ \  - f_i\big( \mathcal{X}_{l-1,m}^{(0,l,-k)}(\mathcal{R}_{t}^{(0,l,k)}), Z^{(0,l,k)} \big) \Big]  \bigg)
+ \mathrm{Var} \bigg(\frac{t}{m^n} \sum_{k=1}^{m^n}  f_i\big(\xi,Z^{(0,0,k)}\big)\bigg) \Bigg] \\
& =   \sum_{l=1}^{n-1}  \Bigg[ \sum_{i=1}^{d} 
  \sum_{k=1}^{m^{n-l}}  \mathrm{Var} \bigg( \frac{t}{m^{n-l}}\Big[  f_i \big( \mathcal{X}_{l,m}^{(0,l,k)} (\mathcal{R}_{t}^{(0,l,k)}), Z^{(0,l,k)} \big) \\
&\ \ \  - f_i\big( \mathcal{X}_{l-1,m}^{(0,l,-k)}(\mathcal{R}_{t}^{(0,l,k)}), Z^{(0,l,k)} \big) \Big]  \bigg) \Bigg]
+ \sum_{i=1}^d  \Bigg[   \mathrm{Var} \bigg(\frac{t}{m^n} \sum_{k=1}^{m^n}  f_i\big(\xi,Z^{(0,0,k)}\big)\bigg) \Bigg] \\
&\leq    \sum_{l=1}^{n-1} \frac{t^2}{m^{n-l}} \E \big[  \| F \big( \mathcal{X}_{l,m}^0 (\mathcal{R}_{t}^0), Z^0 \big)  - F\big( \mathcal{X}_{l-1,m}^1(\mathcal{R}_{t}^0), Z^0 \big) \| ^2  \big] + 
\frac{T^2}{m^n}
 \E \big[ \| F\big(\xi,Z^0 \big) \| ^2 \big].
\end{split}
\raisetag{130pt}
\end{equation}
Furthermore, note that \eqref{lipsch}, the fact that for all $x,y \in \R^d$ it holds that  $\| x+ y \|^2 \leq 2(\| x \|^2 +\| y \|^2)$, items~\eqref{it:propMLP:item1}, \eqref{it:propMLP:item2}, and \eqref{it:propMLP:item4} in \cref{propMLP},
the hypothesis that $(Z^\theta)_{\theta \in \Theta}$ are i.i.d., the hypothesis that ${(\mathfrak{r}^\theta)}_{\theta \in \Theta}$ are i.i.d., the hypothesis that $(Z^\theta)_{\theta \in \Theta}$ and
${(\mathfrak{r}^\theta)}_{\theta \in \Theta}$ are independent,
and \cref{l3.5} 
assure that  for all $t \in [0,T]$, $n \in \N$ it holds that
\begin{equation}
\begin{split}
&  \sum_{l=1}^{n-1} \frac{t^2}{m^{n-l}} \E \big[ \| F \big( \mathcal{X}_{l,m}^0 (\mathcal{R}_{t}^0), Z^0 \big)  - F\big( \mathcal{X}_{l-1,m}^1(\mathcal{R}_{t}^0), Z^0 \big) \| ^2  \big] \\
&\leq  \sum_{l=1}^{n-1} \frac{t^2}{m^{n-l}}  \E \big[ L^2
  \|  \mathcal{X}_{l,m}^0 (\mathcal{R}_{t}^0)  -  \mathcal{X}_{l-1,m}^1(\mathcal{R}_{t}^0) \| ^2  \big]\\
 & \leq  \sum_{l=1}^{n-1} \frac{2L^2 t^2}{m^{n-l}} \big( \E \big[
  \|  \mathcal{X}_{l,m}^0 (\mathcal{R}_{t}^0)  -  X(\mathcal{R}_{t}^0) \| ^2  \big] + \E \big[
  \|  \mathcal{X}_{l-1,m}^0 (\mathcal{R}_{t}^0)  -  X(\mathcal{R}_{t}^0) \| ^2  \big] \big). \\
\end{split}
\end{equation}
The hypothesis that $(\mathfrak{r}^\theta)_{\theta \in \Theta}$ are independent, the hypothesis that $(\mathfrak{r}^\theta)_{\theta \in \Theta}$ and $(Z^\theta)_{\theta \in \Theta}$ are independent,
items~\eqref{it:propMLP:item1} and \eqref{it:propMLP:item2} in \cref{propMLP}, and \cref{l3.7}
hence imply that  for all $t \in [0,T]$, $n \in \N$ it holds that
\begin{equation}
\begin{split}
&  \sum_{l=1}^{n-1} \frac{t^2}{m^{n-l}} \E \big[  \| F \big( \mathcal{X}_{l,m}^0 (\mathcal{R}_{t}^0), Z^0 \big)  - F\big( \mathcal{X}_{l-1,m}^1(\mathcal{R}_{t}^0), Z^0 \big) \| ^2  \big] \\
& \leq  \sum_{l=1}^{n-1} \frac{2L^2 t^2}{m^{n-l}} \big( \E \big[
  \|  \mathcal{X}_{l,m}^0 (\mathcal{R}_{t}^0)  -  X(\mathcal{R}_{t}^0) \| ^2  \big] + \E \big[
  \|  \mathcal{X}_{l-1,m}^0 (\mathcal{R}_{t}^0)  -
  X(\mathcal{R}_{t}^0) \| ^2  \big] \big) \\
  &\leq \sum_{l=1}^{n-1} \frac{2L^2 t}{m^{n-l}} \big(t\, \E \big[
  \|  \mathcal{X}_{l,m}^0 (\mathcal{R}_{t}^0)  -  X(\mathcal{R}_{t}^0) \| ^2  \big] + t\,\E \big[
  \|  \mathcal{X}_{l-1,m}^0 (\mathcal{R}_{t}^0)  -
  X(\mathcal{R}_{t}^0) \| ^2  \big] \big) \\
 & \leq  \sum_{l=1}^{n-1} \frac{2L^2 t}{m^{n-l}} \left(\int_0^t \E \big[ \|  \mathcal{X}_{l,m}^0 (r)  -  X(r) \| ^2  \big]\, dr + \int_0^t 
  \E \big[ \|  \mathcal{X}_{l-1,m}^0 (r)  -  X(r)\| ^2  \big]\, dr \right) \\
  & \leq  \sum_{l=1}^{n-1} \frac{2L^2 T}{m^{n-l}} \left( \int_0^t \E \big[ \|  \mathcal{X}_{l,m}^0 (r)  -  X(r) \| ^2  \big]\, dr + \int_0^t 
  \E \big[ \|  \mathcal{X}_{l-1,m}^0 (r)  -  X(r) \| ^2  \big]\, dr \right).
\end{split}
\raisetag{110pt}
\end{equation}
This and \eqref{eq4} imply that  for all $t \in [0,T]$, $n \in \N$ it holds that
\begin{equation}
\begin{split}
 & \E\big[ \|{\mathcal{X}}_{n,m}^0 (t) - \E[{\mathcal{X}}_{n,m}^0 (t)] \|^2 \big] \\
 & \leq   \sum_{l=1}^{n-1} \frac{t^2}{m^{n-l}} \E \big[  \| F \big( \mathcal{X}_{l,m}^0 (\mathcal{R}_{t}^0), Z^0 \big)  - F\big(  \mathcal{X}_{l-1,m}^1(\mathcal{R}_{t}^0), Z^0 \big) \| ^2  \big] + \frac{T^2}{m^n} \big(
 \E \big[ \| F\big(\xi,Z^0 \big) \| ^2 \big] \big)\\
 & \leq   \sum_{l=1}^{n-1} \frac{2L^2 T}{m^{n-l}} \left(\int_0^t \E \big[
  \|  \mathcal{X}_{l,m}^0 (r)  -  X(r) \| ^2  \big]\, dr + \int_0^t 
  \E \big[
  \|  \mathcal{X}_{l-1,m}^0 (r)  -  X(r) \| ^2  \big]\, dr \right)
 + \frac{T^2}{m^n}
 \E \big[ \| F\big(\xi,Z^0 \big) \| ^2 \big].
\end{split}
\end{equation}
Combining this, e.g.,
the bias-variance type decomposition of the mean square error in   \cite[Lemma 2.2]{jentzen2018lower}, \eqref{eq0}, \eqref{Csquared}, and
item~\eqref{it:expMLP:item1} in \cref{expMLP} demonstrates that  for all $t \in [0,T]$, $n \in \N$ it holds that
\begin{equation} \label{alltogether}
\begin{split}
 &\E \big[ \| {\mathcal{X}}_{n,m}^0 (t) - X(t) \| ^2 \big]
 = \E \big[  \| {\mathcal{X}}_{n,m}^0 (t) - {\E} [{\mathcal{X}}_{n,m}^0(t)] \|^2 \big]  +
  \| {\E} [{\mathcal{X}}_{n,m}^0(t)] - X(t) \|^2 \\[1ex]
 & \leq \sum_{l=1}^{n-1} \frac{2L^2 T}{m^{n-l}} \left(\int_0^t \E \big[ \|  \mathcal{X}_{l,m}^0 (r)  -  X(r) \| ^2  \big]\, dr + \int_0^t 
  \E \big[
 \|  \mathcal{X}_{l-1,m}^0 (r)  -  X(r) \| ^2  \big]\, dr \right)\\
 &  \ \ \ + \frac{T^2}{m^n} 
 \E \big[ \| F\big(\xi,Z^0 \big) \| ^2 \big] + L^2 T  \int_0^t \E \big[  \|  \mathcal{X}_{n-1,m}^0 (r) - X(r) \| ^2   \big]\,dr  \\
  &\leq
     \sum_{l=1}^{n-1} \frac{2L^2 T}{m^{n-l}} \left(\int_0^t \E \big[
  \|  \mathcal{X}_{l,m}^0 (r)  -  X(r) \| ^2  \big]\, dr \right)
  + 
  \sum_{l=0}^{n-2} \frac{2L^2 T}{m^{n-(l+1)}} \left(\int_0^t \E \big[
  \|  \mathcal{X}_{l,m}^0 (r)  -  X(r) \| ^2  \big]\, dr \right)
  \\
  & \quad 
  +  \frac{C^2}{m^n} 
  +   L^2 T  \int_0^t \E \big[  \|  \mathcal{X}_{n-1,m}^0 (r) - X(r) \| ^2   \big]\, dr \\
  & \leq   \sum_{l=1}^{n-1}
   \frac{2L^2 T}{m^{n-l}} \left( \int_0^t \E \big[
  \|  \mathcal{X}_{l,m}^0 (r) -  X(r) \| ^2  \big]\, dr \right)
  + \sum_{l=0}^{n-1}
   \frac{2L^2 T}{m^{n-(l+1)}} \left( \int_0^t \E \big[
 \|  \mathcal{X}_{l,m}^0 (r) -  X(r) \| ^2  \big]\, dr \right) + \frac{C^2}{m^n} \\
  & \leq 
  \frac{C^2}{m^n} +   \sum_{l=0}^{n-1}
   \frac{4L^2 T}{m^{n-(l+1)}} \left( \int_0^t \E \big[
 \|  \mathcal{X}_{l,m}^0 (r) -  X(r) \| ^2  \big]\, dr \right). 
\end{split}
\raisetag{190pt}
\end{equation}
Next let $\epsilon_{n,k} \in [0,\infty]$, $n,k \in \N_0$, satisfy for all $n,k \in \N_0$ that 
\begin{equation} \label{epsilon0}
 \epsilon_{n,0} = \E \big[ \| {\mathcal{X}}_{n,m}^0 (T) - X(T) \| ^2 \big] \qquad\mathrm{and}\qquad
 \epsilon_{n,k+1} = \frac{1}{T^{k+1}} \int_0^T \frac{(T-t)^k}{k!} \E \big[ \| {\mathcal{X}}_{n,m}^0 (t) - X(t) \| ^2 \big]\, dt.
\end{equation}
 Note that \eqref{alltogether} and \eqref{epsilon0} imply that for all $n \in \N$ it holds that
\begin{equation}\label{n,0}
\begin{split}
 \epsilon_{n,0} & = \E \big[ \| {\mathcal{X}}_{n,m}^0 (T) - X(T) \| ^2 \big] 
 \\
 & \leq 
 \frac{C^2}{m^n} +   \sum_{l=0}^{n-1}
 \frac{4L^2 T^2}{m^{n-(l+1)}T} \left( \int_0^T \E \big[
 \|  \mathcal{X}_{l,m}^0 (r) -  X(r) \| ^2  \big]\,dr \right) 
 = \frac{C^2}{m^n(0!)} + 4L^2 T^2  \sum_{l=0}^{n-1}
 \frac{\epsilon_{l,1}}{m^{n-(l+1)}} .
\end{split}
\end{equation}
Moreover, observe that \cref{l3.14} (applied for every $n \in \N_0$, $k \in \N$ with $T \curvearrowleft T$, $k\curvearrowleft k$,
$(U(r))_{r \in [0,T]} \curvearrowleft ( \E[
  \|  \mathcal{X}_{n,m}^0 (r)$ $- X(r) \|^2 ])_{r \in [0,T]}$ in the notation of \cref{l3.14}) and \eqref{epsilon0} demonstrate that for all $n \in \N_0$, $k \in \N$
  it holds that
\begin{equation}
\begin{split}
 &\frac{1}{T^k}\left(\int_0^T \frac{(T-t)^{k-1}}{(k-1)!}
 \int_0^t \E \big[
  \|  \mathcal{X}_{n,m}^0 (r) -  X(r) \| ^2  \big]\, dr\, dt \right)\\
 &= \frac{T}{T^{k+1}} \left( \int_0^T \frac{(T-t)^k}{k!} \E \big[
  \|  \mathcal{X}_{n,m}^0 (t) -  X(t) \| ^2  \big] \,dt
  \right) = T \epsilon_{n,k+1}.
  \end{split}
\end{equation}
This, \eqref{alltogether}, and \eqref{epsilon0} imply that for all $n, k \in \N$ it holds that
\begin{equation} \label{n,k}
\begin{split}
  \epsilon_{n,k} & = \frac{1}{T^{k}} \left( \int_0^T
  \frac{(T-t)^{(k-1)}}{(k-1)!} \E \big[ \| {\mathcal{X}}_{n,m}^0 (t) - X(t) \| ^2 \big]\, dt \right) \\
  & \leq
  \frac{C^2}{T^k m^n}  \left( \int_0^T
  \frac{(T-t)^{(k-1)}}{(k-1)!} \, dt  \right) 
  + \sum_{l=0}^{n-1}
   \frac{4L^2 T}{m^{n-(l+1)}T^k} \left(  \int_0^T
  \frac{(T-t)^{(k-1)}}{(k-1)!} \int_0^t \E \big[
  \|  \mathcal{X}_{l,m}^0 (r) -  X(r) \| ^2  \big]\, dr \, dt\right) \\ 
  & =  \frac{C^2}{T^k m^n} \frac{T^k}{k!} + 4L^2T\left[\sum_{l=0}^{n-1} \frac{T \epsilon_{l,k+1}}{m^{n-(l+1)}}
  \right] 
  =
  \frac{C^2}{m^n (k!)} + 4L^2T^2\left[\sum_{l=0}^{n-1} \frac{ \epsilon_{l,k+1}}{m^{n-(l+1)}}
  \right].
\end{split}
\end{equation}
Furthermore, note that \cref{l3.3} and \eqref{C} prove that for all $t \in [0,T]$ it holds that 
\begin{equation}\label{C^2inequal}
\left\| X(t) - \xi \right\|^2 \leq T^2  {\E}  \big[ \| F(\xi, Z^0) \|^2 \big]  e^{2LT}= C^2.
\end{equation}
The fact that for all $t \in [0,T]$ it holds that $\mathcal{X}^0_{0,m}(t) = \xi$ and \eqref{epsilon0} hence assure that 
\begin{equation} \label{0,0}
\epsilon_{0,0} =\E \big[ \| X(T) - \xi \| ^2 \big]
=  \left\| X(T) - \xi \right\| ^2 \leq C^2 = 
\frac{C^2}{m^0 0!}.
\end{equation}
Moreover, observe that \eqref{epsilon0}, \eqref{C^2inequal}, and the fact that for all $t \in [0,T]$ it holds that $\mathcal{X}^0_{0,m}(t) = \xi$ ensure that for all $k \in \N$
it holds that 
\begin{equation}
\begin{split}
&\epsilon_{0,k} = \frac{1}{T^{k}} \int_0^T
\frac{(T-t)^{(k-1)}}{(k-1)!} \E \big[ \| X(t) - \xi \| ^2 \big]\, dt
\leq
\frac{C^2}{T^{k}} \int_0^T
\frac{(T-t)^{(k-1)}}{(k-1)!} \, dt 
= \frac{C^2}{T^k} \frac{T^k}{k!} = \frac{C^2}{m^0 (k!)}.
\end{split}
\end{equation}
Combining this, \eqref{n,0}, \eqref{n,k}, and \eqref{0,0} demonstrates that for all $n,k \in \N_0$ it holds that
\begin{equation}
\begin{split}
 \epsilon_{n,k}
 & \leq \frac{C^2}{m^n (k!)} + 4L^2T^2\left[\sum_{l=0}^{n-1} \frac{ \epsilon_{l,k+1}}{m^{n-(l+1)}} \right] \\
 & =  \frac{C^2 m^k}{m^{n+k}(k!)} + 4L^2T^2
 \left[ \sum_{l=0}^{n-1} \frac{ \epsilon_{l,k+1}}{m^{n-(l+1)}}
  \right] 
  \leq
   \frac{C^2 e^m}{m^{n+k}} + 4L^2T^2
\left[ \sum_{l=0}^{n-1} \frac{ \epsilon_{l,k+1}}{m^{n-(l+1)}}
  \right].
\end{split}
\end{equation}
\cref{l3.12} (applied with $\alpha \curvearrowleft C^2 e^m$, $\beta \curvearrowleft 4L^2T^2$, $M \curvearrowleft m$, $(\epsilon_{n,k})_{n,k \in \N_0} \curvearrowleft (\epsilon_{n,k})_{n,k \in \N_0}$ in the notation of \cref{l3.12})
therefore proves that for all $n,k \in \N_0$ it holds that
\begin{equation}
\epsilon_{n,k} \leq 
\frac{C^2 e^m (1+4L^2T^2)^n}{m^{n+k}}.
\end{equation}
This and \eqref{epsilon0} imply that for all $n \in \N_0$ it holds that
\begin{equation}
 {\E} \big[ \| X(T) - {\mathcal{X}}_{n,m}^0(T) \|^2\big] 
= \epsilon_{n,0} 
\leq \frac{C^2 e^m (1+4L^2T^2)^n}{m^{n}}.
\end{equation}
The fact that for all $x,y \in [0,\infty)$ it holds that 
$\sqrt{x+y} \leq \sqrt{x} + \sqrt{y}$ and \eqref{C}  hence demonstrate
that for all $n \in \N_0$ it holds that
\begin{equation}
\begin{split}
\big( {\E} \big[ \| X(T) - {\mathcal{X}}_{n,m}^0(T) \|^2\big] 
\big)^{1/2}
  &\leq 
 \frac{C e^{m/2} (\sqrt{1+4L^2T^2})^n}{m^{n/2}}
 \leq 
 \frac{C(1+2LT)^n e^{m/2}}{m^{n/2}} 
 \\
 & =  \frac{ T
 \big( {\E}  \big[ \|F(\xi, Z^0)\| ^2 \big] \big) ^{1/2}(1+2LT)^n e^{(LT + m/2)}}{m^{n/2}} .
\end{split}
\end{equation}
The proof of \cref{thmerrorestimate} is thus completed.
\end{proof}

\section{Complexity analysis for MLP approximation algorithms}\label{section3}

In this section we provide in \cref{THEOREM} in \cref{subsectionoverallcompl} below an overall complexity analysis for the MLP approximations introduced in \cref{setting} in \cref{subsectionsetting} above.

Our proof of \cref{THEOREM} combines the error analysis in \cref{thmerrorestimate} with the bound for the computational cost of the proposed MLP approximations \cref{lemma1} in \cref{subsection3.1} and the elementary results in \crefrange{l4.3}{Lemma2.4}.
Proofs for \cref{lemma1} and \cref{l4.3} can be found, e.g., in \cite[Lemma 3.14]{beck2020overcoming} and \cite[Lemma 4.3]{beckAllenCahn} and \cref{Lemma2.4} is a slightly modified version of \cite[Lemma 4.2]{beckAllenCahn}.

\subsection{Computational cost analysis for MLP approximation algorithms}\label{subsection3.1}

\begin{lemma} \label{lemma1}
Let ${(RV_{n,m})}_{n \in \N_0, m \in \N} \subseteq {\N}_0$ satisfy for all $n, m \in {\N}$ that 
$RV_{0,m} = 0$ and
$RV_{n,m} \leq m^n + \sum_{l=1}^{n-1} [ m^{n-l}( 1 + RV_{l,m} + RV_{l-1,m}) ]$.
Then it holds for all $n,m \in {\N}$ that $RV_{n,m}\leq {(3m)}^n$.
\end{lemma}

\begin{lemma}\label{l4.3}
	Let $\alpha \in [1, \infty)$. Then it holds for all $k\in \N$ that 
	$\sum_{n =1}^k (\alpha n)^n \leq  2(\alpha k)^{k}$.
\end{lemma}

\subsection{Overall complexity analysis for MLP approximation algorithms}\label{subsectionoverallcompl}

\begin{lemma} \label{finiteexp}
Let $d \in {\N}$, $T, L \in [0,\infty)$, $\xi \in \R^d$, $ X \in \mathcal{C}([0,T],\R^d)$,
let $ \left\|  \cdot \right\| \! \colon \R^d \to [0,\infty)$ be a norm on $\R^d$,
 let $(S,\mathcal{S})$ be a
measurable space, let $F\colon{\R}^d\times
S\rightarrow {\R}^d$ be $(\mathcal{B}({\R}^d) \otimes \mathcal{S}) / \mathcal{B}({\R}^d)$-measurable,
assume for all $x,y \in {\R}^d$, $s \in S$ that 
$ \| F(x,s)-F(y,s)  \| \leq L \| x-y \| $,
 let $(\Omega,\mathcal{F},{\P})$ be a probability space, 
let $Z \colon\Omega\rightarrow S$ 
be a random variable, and assume that $\E[ \|F(\xi, Z) \|^2] < \infty$.
Then 
\begin{enumerate}[(i)]
\item\label{it:finiteexp:item1} for all $t \in [0,T]$ it holds that
 $[0,t] \times \Omega \ni (r, \omega)
\mapsto F(X(r), Z(\omega)) \in \R^d$ is
$(\mathcal{B}([0,t]) \otimes \mathcal{F})/\mathcal{B}(\R^d)$-measurable,
\item\label{it:finiteexp:item2} for all $t \in [0,T]$ it holds that $\E[ \|F(X(t),Z)\|]< \infty$, and
\item\label{it:finiteexp:item3} for all $t \in [0,T]$ it holds that 
$ \int_0^t{\E[ \| F (X(r), Z) \|]} \, dr < \infty $.
\end{enumerate}
\end{lemma}

\begin{proof}[Proof of \cref{finiteexp}]
Throughout this proof let $t \in [0,T]$.
Note that the hypothesis that for all $x,y \in {\R}^d$, $s \in S$ it holds that 
$ \| F(x,s)-F(y,s)  \| \leq L \| x-y \| $, the hypothesis that 
$F\colon{\R}^d\times
S\rightarrow {\R}^d$ is $(\mathcal{B}({\R}^d) \otimes \mathcal{S}) / \mathcal{B}({\R}^d)$-measurable,
the hypothesis that $Z$ is $\mathcal{F}/\mathcal{S}$-measurable,
and the hypothesis that $ X \in \mathcal{C}([0,T],\R^d)$
assure that 
$[0,t] \times \Omega \ni (r, \omega)
\mapsto F(X(r), Z(\omega)) \in \R^d$ is
$(\mathcal{B}([0,t]) \otimes \mathcal{F})/\mathcal{B}(\R^d)$-measurable.
This establishes item~\eqref{it:finiteexp:item1}.
In addition, note  that the assumption that for all $x,y \in {\R}^d$, $s \in S$ it holds that 
$ \| F(x,s)-F(y,s)  \| \leq L \| x-y \| $, and the triangle inequality ensure that for all  $x \in {\R}^d$, $s \in S$ it holds that
\begin{equation} \label{lip}
\| F(x,s) \|  \leq \| F(\xi,s) \| + \| F(x,s) -F(\xi,s) \| \leq 
 \| F(\xi,s) \| + L \| x-\xi \|.
\end{equation}
In addition, observe that the hypothesis that $\E[ \|F(\xi, Z) \|^2] < \infty$ and Jensen's inequality assure that 
\begin{equation}
(\E[\|F(\xi,Z) \|])^2 \leq \E [ \|F(\xi, Z) \|^2] < \infty.
\end{equation}
Combining this, \eqref{lip}, and the hypothesis that $ X \in \mathcal{C}([0,T],\R^d)$ implies that 
\begin{equation}
\label{meanfinito}
{{\E}[ \| F (X(t), Z) \| ]} \leq \E[\| F(\xi,Z) \| + L \| X(t)-\xi \|] = \E[\| F(\xi,Z) \| ] + L  \| X(t)-\xi \| < \infty.
\end{equation}
This establishes item~\eqref{it:finiteexp:item2}.
Moreover, observe that \eqref{meanfinito}, the assumption that for all $x,y \in {\R}^d$, $s \in S$ it holds that 
$ \| F(x,s)-F(y,s)  \| \leq L \| x-y \| $, and the fact that for all $a,b \in \R^d$ it holds that $\|a\| - \|b\| \leq \|a-b\|$ 
assure that 
for all $r,s \in [0,t]$ it holds that
\begin{equation}
    \begin{split}
 &| \E[\| F(X(r),Z) \|]-\E[\| F(X(s),Z) \| ]| 
= |\E [\| F(X(r),Z)\| - \| F(X(s),Z) \|]| \\
 & \leq |\E [\| F(X(r),Z) -  F(X(s),Z) \|] | 
 \leq |\E [L \| X(r) -  X(s) \|]| 
 = L   \| X(r) -  X(s) \| .
    \end{split}
\end{equation}
Combining this with the hypothesis that $ X \in \mathcal{C}([0,T],\R^d)$ ensures that  $[0,t] \ni r \mapsto \E[\| F(X(r),Z) \|] \in \R $ is continuous.
This ensures that 
$ \int_0^t{\E[ \| F (X(r), Z) \|]} \, dr < \infty $.
This establishes item~\eqref{it:finiteexp:item3}.
This completes the proof of \cref{finiteexp}.
\end{proof}

\begin{lemma}\label{Lemma2.4}
Let $T,L,C \in [0,\infty)$, $\alpha \in [1,\infty)$, $\mathfrak{N} \in \N_0$, $(\gamma_n)_{n \in \N} \subseteq [0,\infty)$,
$(\epsilon_n)_{n \in \N} \subseteq [0,\infty)$,
assume for all $n \in \N$ that $\gamma_n \leq (\alpha n)^n$ and $\epsilon_n \leq C e^{n/2}(1+2LT)^n n^{-n/2}$,
and let $N=(N_{\varepsilon})_{\varepsilon\in (0,1]}  \colon (0, 1]  \rightarrow [0,\infty]$ satisfy for all $\varepsilon \in (0, 1]$ that
\begin{equation}\label{Nepsilon}
    N_\varepsilon =   \min\! \Bigg( \left\lbrace   n \in \N \colon \sup_{m \in [n,\infty)\cap\N} \Big[ C e^{m/2}(1+2LT)^m m^{-m/2} \Big] \leq \varepsilon \right\rbrace  \cup \lbrace \infty \rbrace\Bigg).
\end{equation}
Then  
\begin{enumerate}[(i)]
\item\label{it:Lemma2.4:item1} for all $\varepsilon \in (0,1]$ it holds that $N_\varepsilon < \infty$ and
\item\label{it:Lemma2.4:item2} for all $\varepsilon \in (0,1]$, $\delta \in (0, \infty)$ it holds that
 $\sup_{n \in  [N_\varepsilon, \infty) \cap \N} 
    \epsilon_n \leq \varepsilon$  and
\begin{equation}
\begin{split}
\sup_{n \in  [1,N_\varepsilon + \mathfrak{N}] \cap \N}  \gamma_n
\leq & 
\sup_{n \in \N} 
  \left[ \frac{[\alpha(n + \mathfrak{N}+1)]^{(n + \mathfrak{N}+1)}}{n^{n(1+\delta)}}
  e^{n(1+\delta)}
  (1+2LT)^{n(2+2\delta)} \right] 
  \max \lbrace 1, C^{2+2\delta} \rbrace \varepsilon^{-(2+2\delta)}
    < \infty.
 \end{split}
\end{equation}
\end{enumerate}
\end{lemma}

\begin{proof}[Proof of \cref{Lemma2.4}]
Throughout this proof let $\mathfrak{a}_\delta \in [0,\infty]$, $\delta \in (0, \infty)$, satisfy 
for all $\delta \in (0, \infty)$ that
\begin{equation} \label{adelta}
    \mathfrak{a}_\delta = 
C^{2+2\delta} \sup_{n \in \N} 
  \left[ \frac{[\alpha(n + \mathfrak{N}+1)]^{(n + \mathfrak{N}+1)}}{n^{n(1+\delta)}}
  e^{n(1+\delta)}
  (1+2LT)^{n(2+2\delta)} \right].
\end{equation}
Observe that the fact that for all $t \in (0,\infty)$ it holds that
$\mathrm{ln}(t) \leq t-1$ ensures that
\begin{equation}
    \begin{split}
        & \limsup_{n \rightarrow\infty} 
        \big[\mathrm{ln}\big( Ce^{n/2}(1+2LT)^n n^{-n/2}\big)\big] = \limsup_{n \rightarrow\infty} 
        \big[\mathrm{ln}(C) + \frac{n}{2}
        +n\mathrm{ln}(1+2LT) -
        \frac{n}{2}\mathrm{ln}(n)\big]\\
        & \leq \limsup_{n \rightarrow\infty} 
        \big[\mathrm{ln}(C) + \frac{n}{2}
        +n2LT -
        \frac{n}{2}\mathrm{ln}(n)\big]
        = \limsup_{n \rightarrow\infty} 
        \Big[n\mathrm{ln}(n) \Big(\frac{\mathrm{ln}(C)}{n\mathrm{ln}(n)} + \frac{1}{2\mathrm{ln}(n)} +
        \frac{2LT}{\mathrm{ln}(n)} -
        \frac{1}{2} \Big) \Big] = - \infty.
    \end{split}
\end{equation}
This and the fact that $\lim_{t\rightarrow-\infty}e^t = 0$ imply that
\begin{equation}
    \begin{split}
        0 \leq \limsup_{n \rightarrow\infty} 
        \big[ Ce^{n/2}(1+2LT)^n n^{-n/2}\big] 
        = \limsup_{n \rightarrow\infty}
        \big[\exp \big(\mathrm{ln}\big(Ce^{n/2}(1+2LT)^n n^{-n/2}\big) \big) \big] 
        =0.
    \end{split}
\end{equation}
This and \eqref{Nepsilon} ensure that for all $\varepsilon \in (0,1]$ it holds that $N_\varepsilon < \infty$.
This establishes item~\eqref{it:Lemma2.4:item1}.
Next note that the fact that for all $t \in (0,\infty)$ it holds that
$\mathrm{ln}(t) \leq t-1$ ensures that for all $\delta \in (0, \infty)$ it holds that
\begin{equation}
    \begin{split}
        & \limsup_{n \rightarrow\infty} 
        \left[\mathrm{ln}\bigg(\frac{[\alpha(n + \mathfrak{N}+1)]^{(n + \mathfrak{N}+1)}}{n^{n(1+\delta)}}
  e^{n(1+\delta)}
  (1+2LT)^{n(2+2\delta)}\bigg) \right]\\
        & = \limsup_{n \rightarrow\infty} 
        \big[(n+\mathfrak{N}+1)\mathrm{ln}(\alpha) + (n+\mathfrak{N}+1)\mathrm{ln}(n+\mathfrak{N}+1)
        - n(1+\delta)\mathrm{ln}(n) +
        n(1+ \delta)
        +n(2+ 2\delta)\mathrm{ln}(1+2LT) \big]\\
        & \leq \limsup_{n \rightarrow\infty}
        \bigg[n\mathrm{ln}(n) \bigg( \frac{(n+\mathfrak{N}+1)\mathrm{ln}(\alpha)}{n\mathrm{ln}(n)} 
        +\frac{(n+\mathfrak{N}+1)}{n}\frac{\mathrm{ln}(n+\mathfrak{N}+1)}{\mathrm{ln}(n)}
        - (1+\delta) + \frac{(1+\delta)}{\mathrm{ln}(n)} + \frac{(2+ 2\delta)2LT}{\mathrm{ln}(n)}
        \bigg) \bigg] 
        \\
        & = - \infty.
    \end{split}
\end{equation}
This and the fact that $\lim_{t\rightarrow-\infty}e^t = 0$ imply that for all $\delta \in (0, \infty)$ it holds that
\begin{equation}
    \begin{split}
        & 0 \leq \limsup_{n \rightarrow\infty} 
        \left[\frac{[\alpha(n + \mathfrak{N}+1)]^{(n + \mathfrak{N}+1)}}{n^{n(1+\delta)}}
  e^{n(1+\delta)}
  (1+2LT)^{n(2+2\delta)} \right] \\
        & = \limsup_{n \rightarrow\infty}
        \left[\exp \bigg(\mathrm{ln}\bigg(\frac{[\alpha(n + \mathfrak{N}+1)]^{(n + \mathfrak{N}+1)}}{n^{n(1+\delta)}}
  e^{n(1+\delta)}
  (1+2LT)^{n(2+2\delta)}\bigg) \bigg) \right]
  =0.
    \end{split}
\raisetag{50pt}
\end{equation}
Combining this and \eqref{adelta} implies that 
for all $\delta \in (0, \infty)$ it holds that
\begin{equation}\label{finitea}
        \mathfrak{a}_\delta 
        = 
C^{2+2\delta}\sup_{n \in \N} 
  \left[ \frac{[\alpha(n + \mathfrak{N}+1)]^{(n + \mathfrak{N}+1)}}{n^{n(1+\delta)}}
  e^{n(1+\delta)}
  (1+2LT)^{n(2+2\delta)} \right] < \infty.
\end{equation}
In addition, observe that \eqref{Nepsilon}, the fact that for all $\varepsilon \in (0,1]$ it holds that $N_\varepsilon < \infty$, and the hypothesis that for all $n \in \N$ it holds that $\epsilon_n \leq C e^{n/2}(1+2LT)^n n^{-n/2}$ 
assure that for all $\varepsilon \in (0, 1]$ it holds that
\begin{equation} \label{260}
\sup_{n \in  [N_\varepsilon, \infty) \cap \N} \epsilon_n \leq
\sup_{n \in  [N_\varepsilon, \infty) \cap \N}
\big[C e^{n/2}(1+2LT)^n n^{-n/2} \big]
  \leq \varepsilon.
\end{equation}
Next let $E \subseteq (0,1]$ satisfy that $E = \lbrace \varepsilon \in (0,1] \colon N_\varepsilon >1
\rbrace$.
Note that \eqref{Nepsilon} ensures that 
for all $\varepsilon \in E$ 
it holds that
\begin{equation}
Ce^{{(N_{\varepsilon}-1)}/2}(1+2LT)^{(N_{\varepsilon}-1)}  (N_{\varepsilon}-1)^{-(N_{\varepsilon}-1)/2} > \varepsilon.
\end{equation} 
This implies that for all $\varepsilon \in E$ 
it holds that
\begin{equation}\label{f_Ne-1}
\frac{C}{\varepsilon}e^{{(N_{\varepsilon}-1)}/2}(1+2LT)^{(N_{\varepsilon}-1)}   > (N_{\varepsilon}-1)^{(N_{\varepsilon}-1)/2}
.
\end{equation} 
This, the hypothesis that for all $n \in \N$ it holds that $\gamma_n \leq (\alpha n)^n$, \eqref{finitea}, and the fact that
for all $K \in \N$, $\beta \in [1,\infty)$ it holds that
$\sup_{n \in  [1,K] \cap \N} (\beta n)^n = (\beta K)^K$
imply that for all $\varepsilon \in E$, $\delta \in (0,\infty)$
it holds that
\begin{equation} \label{262}
\begin{split}
\sup_{n \in  [1,N_\varepsilon + \mathfrak{N}] \cap \N}  \gamma_n
&\leq \sup_{n \in  [1,N_\varepsilon + \mathfrak{N}] \cap \N}  (\alpha n)^n 
= [\alpha(N_\varepsilon + \mathfrak{N})]^{(N_\varepsilon + \mathfrak{N})} 
=  \frac{[\alpha(N_\varepsilon + \mathfrak{N})]^{(N_\varepsilon + \mathfrak{N})}}{(N_{\varepsilon}-1)^{(N_{\varepsilon}-1)(1+\delta)}} (N_{\varepsilon}-1)^{(N_{\varepsilon}-1)(1+\delta)}\\
 &\leq \frac{[\alpha(N_\varepsilon + \mathfrak{N})]^{(N_\varepsilon + \mathfrak{N})}}{(N_{\varepsilon}-1)^{(N_{\varepsilon}-1)(1+\delta)}} \frac{C^{2+2\delta}} {\varepsilon^{2+2\delta}}
 e^{(N_\varepsilon-1)(1+\delta)}
  (1+2LT)^{(N_{\varepsilon}-1)(2+2\delta)} \\
  & \leq C^{2+2\delta} \varepsilon^{-(2+2\delta)}
\sup_{n \in [2,\infty)\cap\N} 
  \left[ \frac{[\alpha(n + \mathfrak{N})]^{(n + \mathfrak{N})}}{(n-1)^{(n-1)(1+\delta)}}
  e^{(n-1)(1+\delta)}
  (1+2LT)^{(n-1)(2+2\delta)} \right]\\
  & = 
  C^{2+2\delta} \varepsilon^{-(2+2\delta)}
\sup_{n \in \N} 
  \left[ \frac{[\alpha(n + \mathfrak{N}+1)]^{(n + \mathfrak{N}+1)}}{n^{n(1+\delta)}}
  e^{n(1+\delta)}
  (1+2LT)^{n(2+2\delta)} \right] 
  = \mathfrak{a}_\delta \varepsilon^{-(2+2\delta)}
  < \infty
 .
 \end{split}
\end{equation}
Moreover, observe that the hypothesis that for all $n \in \N$ it holds that $\gamma_n \leq (\alpha n)^n$ and the fact that
for all $K \in \N$, $\beta \in [1,\infty)$ it holds that
$\sup_{n \in  [1,K] \cap \N} (\beta n)^n = (\beta K)^K$
assure that for all 
$\varepsilon \in (0,1] \setminus E$, $\delta \in (0,\infty)$ 
it holds that 
\begin{equation}\label{rvne1}
\begin{split}
\sup_{n \in  [1,N_\varepsilon + \mathfrak{N}] \cap \N}  \gamma_n
& = 
\sup_{n \in  [1,\mathfrak{N}+1] \cap \N}  \gamma_n
\leq \sup_{n \in  [1,\mathfrak{N}+1] \cap \N} (\alpha n)^n
= [\alpha(\mathfrak{N}+1)]^{(\mathfrak{N}+1)}
\\
& \leq  \varepsilon^{-(2+2\delta)} [\alpha(\mathfrak{N}+1)]^{(\mathfrak{N}+1)}
< \infty.
\end{split}
\end{equation}
Next observe that \eqref{finitea} implies that for all $\delta \in (0,\infty)$ it holds that 
\begin{equation} \label{111}
\begin{split}
    \mathfrak{a}_\delta &= 
C^{2+2\delta} \sup_{n \in \N} 
  \left[ \frac{[\alpha(n + \mathfrak{N}+1)]^{(n + \mathfrak{N}+1)}}{n^{n(1+\delta)}}
  e^{n(1+\delta)}
  (1+2LT)^{n(2+2\delta)} \right]\\
  & \leq 
  \max \lbrace 1, C^{2+2\delta} \rbrace
  \sup_{n \in \N} 
  \left[ \frac{[\alpha(n + \mathfrak{N}+1)]^{(n + \mathfrak{N}+1)}}{n^{n(1+\delta)}}
  e^{n(1+\delta)}
  (1+2LT)^{n(2+2\delta)} \right] 
  < \infty.
  \end{split}
\end{equation}
In addition, observe that for all $\delta \in (0,\infty)$ it holds that
\begin{equation}
\begin{split}
  [\alpha(\mathfrak{N}+1)]^{(\mathfrak{N}+1)} 
  &\leq 
    \sup_{n \in \N} 
  \left[ \frac{[\alpha(n + \mathfrak{N}+1)]^{(n + \mathfrak{N}+1)}}{n^{n(1+\delta)}}
  e^{n(1+\delta)}
  (1+2LT)^{n(2+2\delta)} \right]\\
  & \leq 
  \max \lbrace 1, C^{2+2\delta} \rbrace
  \sup_{n \in \N} 
  \left[ \frac{[\alpha(n + \mathfrak{N}+1)]^{(n + \mathfrak{N}+1)}}{n^{n(1+\delta)}}
  e^{n(1+\delta)}
  (1+2LT)^{n(2+2\delta)} \right].
  \end{split}
\end{equation}
Combining this, \eqref{262}, \eqref{rvne1}, and \eqref{111}  establishes that 
for all $\varepsilon \in (0,1]$, $\delta \in (0,\infty)$ it holds that
\begin{equation}
\begin{split}
 \sup_{n \in  [1,N_\varepsilon + \mathfrak{N}] \cap \N}  \gamma_n
 & \leq  \, \varepsilon^{-(2+2\delta)}
 \max\Bigg\lbrace 
 \mathfrak{a}_\delta, [\alpha(\mathfrak{N}+1)]^{(\mathfrak{N}+1)} \Bigg\rbrace \\
  & \leq \varepsilon^{-(2+2\delta)}
  \max \lbrace 1, C^{2+2\delta} \rbrace
  \sup_{n \in \N} 
  \left[ \frac{[\alpha(n + \mathfrak{N}+1)]^{(n + \mathfrak{N}+1)}}{n^{n(1+\delta)}}
  e^{n(1+\delta)}
  (1+2LT)^{n(2+2\delta)} \right] < \infty.
 \end{split}
\end{equation}
This and \eqref{260} establish item~\eqref{it:Lemma2.4:item2}.
The proof of \cref{Lemma2.4} is thus completed.
\end{proof}

\begin{theorem}\label{THEOREM}
Let $d \in \N$, $\mathfrak{N} \in \N_0$, 
$T, L \in [0,\infty)$, $\Theta = \cup_{n=1}^{\infty} {\Z}^n$,  $\xi \in \R^d$, let $ \left\|  \cdot \right\| \! \colon \R^d \to [0,\infty) $ be a norm on $\R^d$, 
  let $(S,\mathcal{S})$ be a
measurable space, 
let $F \colon{\R}^d\times
S\rightarrow {\R}^d$  be $(\mathcal{B}({\R}^d) \otimes \mathcal{S}) / \mathcal{B}({\R}^d)$-measurable,
assume for all $x,y \in {\R}^d$, $s \in S$  that
$ \| F(x,s)-F(y,s)\| \leq L\|x-y\| $, 
 let $(\Omega,\mathcal{F},{\P})$ be a probability space,
 let $Z^\theta \colon\Omega\rightarrow S$, $\theta \in\Theta$,
be i.i.d.\ random variables, 
assume that $\E[ \|F(\xi, Z^0) \|^2] < \infty$,
let $\mathfrak{r}^\theta \colon \Omega \rightarrow [0,1]$,
$\theta \in \Theta$, be independent $\mathcal{U}_{[0,1]}$-distributed random variables,
 let $\mathcal{R}^\theta \colon
[0,T]\times\Omega \rightarrow [0,T]$, $\theta \in \Theta$, satisfy for all $t \in [0,T]$,  $\theta \in \Theta$ that $\mathcal{R}^\theta_t =  \mathfrak{r}^\theta t$,
assume that
${(\mathfrak{r}^\theta)}_{\theta\in\Theta}$ and
${(Z^\theta)}_{\theta\in\Theta}$ are independent,
let ${\mathcal{X}}_{n,m}^{\theta} \colon [0,T]\times\Omega \rightarrow {\R}^d $, $n,m \in \N_0$, $\theta \in \Theta$, satisfy
for all  $ n \in \N_0$,  $m \in {\N}$, $\theta \in \Theta$, $t \in [0,T]$ that
\begin{equation}
 \begin{split}
 {\mathcal{X}}_{n,m}^{\theta} (t) 
 & = \Bigg[
 \sum_{l=1}^{n-1}
 \frac{t}{m^{n-l}}  \sum_{k=1}^{m^{n-l}} F \big( \mathcal{X}_{l,m}^{(\theta,l,k)} (\mathcal{R}_{t}^{(\theta,l,k)}), Z^{(\theta,l,k)} \big) - F \big( \mathcal{X}_{l-1,m}^{(\theta,l,-k)}(\mathcal{R}_{t}^{(\theta,l,k)}), Z^{(\theta,l,k)} \big)\Bigg] \\
 &
 +\Bigg[  \frac{t \mathbbm{1}_{\N}(n)}{m^n}  \sum_{k=1}^{m^n} F \big(\xi,Z^{(\theta,0,k)}\big)\Bigg] 
 + \xi,
 \end{split}
\end{equation}
and let ${(RV_{n,m})}_{n \in \N_0, m \in \N} \subseteq {\N}_0$ satisfy for all $n,m \in {\N}$ that 
$RV_{0,m} = 0$
 and
\begin{equation} \label{againRV}
RV_{n,m} \leq m^n + \sum_{l=1}^{n-1} 
\left[ m^{n-l}( 1 + RV_{l,m} + RV_{l-1,m})  \right].
\end{equation}
Then 
\begin{enumerate}[(i)]
\item\label{it:THEOREM:item1} there exists a unique
$X \in \mathcal{C}([0,T], {\R}^d)$ such that for all $t \in [0,T]$ it holds that
$X(t) = \xi + \int_0^t {{\E}[ F (  X(r), Z^0) ]}\, dr$ (cf.\ \cref{finiteexp}) and
\item\label{it:THEOREM:item2} there exist $c
=(c_{\delta})_{\delta \in (0,\infty)}
\colon (0,\infty) \rightarrow [0,\infty)$ and $N
=(N_\varepsilon)_{\varepsilon \in (0,1]}
\colon (0,1] \rightarrow \N$ such that
for all $\delta \in (0, \infty)$, $\varepsilon \in (0,1]$ it holds that
$\sum_{n=1}^{N_{\varepsilon}+\mathfrak{N}}
RV_{n,n} \leq c_{\delta} {\varepsilon}^{-(2+\delta)}$ and
$\sup_{n \in  [N_\varepsilon, \infty) \cap \N} ( {\E}[\| X(T) - {\mathcal{X}}_{n,n}^{0}(T)\|^2])^{1/2} \leq \varepsilon$.  
\end{enumerate}
\end{theorem}

\begin{proof}[Proof of \cref{THEOREM}]
Throughout this proof 
let $ \vertiii \cdot  \colon \R^d \to [0,\infty) $ be the standard norm on $\R^d$,
let $a,b \in (0,\infty)$ satisfy
for all $x \in \R^d$ that 
	$a \| x \| \leq \vertiii{x} \leq b \| x \|$
(cf., e.g., Kreyszig 
\cite[Theorem 2.4-5]{MR0467220}),
let $C \in [0,\infty)$ satisfy that
$ C =ba^{-1}e^{a^{-1}bLT}T({\E}[ \|F(\xi, Z^0) \|^2])^{1/2} $,  
and let  $N
\colon (0,1] \rightarrow \N \cup \lbrace \infty \rbrace$ satisfy for all 
$\varepsilon \in (0,1]$ that
\begin{equation}
    N_\varepsilon =   \min\! \Bigg( \left\lbrace   n \in \N \colon \sup_{m \in [n,\infty)\cap\N} \Big[ C e^{m/2}(1+2a^{-1}bLT)^m m^{-m/2} \Big] \leq \varepsilon \right\rbrace  \cup \lbrace \infty \rbrace\Bigg).
\end{equation}
Observe that the assumption that for all
	$x,y \in {\R}^d$, $s \in S$ 
it holds that
	$ \| F(x,s)-F(y,s)\| \leq L\|x-y\| $
ensures that for all $x,y \in \R^d$ it holds that
\begin{equation}
    \begin{split}
        \| \E[F(x,Z^0)] - \E[F(y,Z^0)] \| 
        &=\| \E[F(x,Z^0) - F(y,Z^0)] \| 
        \\& 
        \leq  \E[\|F(x,Z^0) - F(y,Z^0)\|] 
        \leq L\E[\|x-y\|]
        = L\|x-y\|.
    \end{split}
\end{equation}
This ensures that
 there exists a unique
$X \in \mathcal{C}([0,T], {\R}^d)$ such that for all $t \in [0,T]$ it holds that
\begin{equation}
X(t) = \xi + \int_0^t {{\E}\big[ F (  X(r), Z^0)  \big]}\, dr 
\end{equation}
(cf., e.g., Teschl \cite[Section 2]{MR2961944}). 
This establishes item~\eqref{it:THEOREM:item1}.
Next observe that item~\eqref{it:Lemma2.4:item1} in \cref{Lemma2.4} (applied with $T \curvearrowleft T$, $L \curvearrowleft a^{-1}bL$, $C \curvearrowleft C$, $\mathfrak{N} \curvearrowleft \mathfrak{N}$ in the notation of \cref{Lemma2.4})  assures that for all 
$\varepsilon \in (0,1]$ it holds that $ N_\varepsilon< \infty$.
Next note that the fact that for all $x \in \R^d$ it holds that 
$a \| x \| \leq \vertiii{x} \leq b \| x \|$ and the assumption that for all $x,y \in {\R}^d$, $s \in S$ it holds that
$ \| F(x,s)-F(y,s)\| \leq L\|x-y\| $ ensure that for all $x,y \in \R^d$, $s \in S$ it holds that 
\begin{equation}
   \vertiiii{F(x,s)-F(y,s)} \leq b\|F(x,s)-F(y,s)\| \leq 
   bL \| x-y \| \leq 
   bLa^{-1}\vertiiii{x-y}.
\end{equation}
This, the fact that for all $x \in \R^d$ it holds that 
$a \| x \| \leq \vertiii{x} \leq b \| x \|$, the fact that
$ C =ba^{-1}e^{a^{-1}bLT}T({\E}[ \|F(\xi, Z^0) \|^2])^{1/2} $, \eqref{againRV}, \cref{lemma1}, and \cref{thmerrorestimate} imply that for all $n \in \N$ it holds that
$RV_{n,n} \leq (3n)^n$ and 
\begin{equation}
\begin{split}
&\big( {\E} \big[\| X(T) - {\mathcal{X}}_{n,n}^0(T)\|^2\big] 
\big)^{1/2} 
 \leq \frac{1}{a} \big( {\E} \big[\vertiiii{ X(T) - {\mathcal{X}}_{n,n}^0(T)}^2\big]
\big)^{1/2} \\ 
&\leq 
\frac{1}{a} \frac{ T
 \big( {\E}  \big[ \vertiiii{F(\xi, Z^0)}^2 \big] \big) ^{1/2}(1+2a^{-1}bLT)^{n} e^{(a^{-1}bLT + n/2)}}{n^{n/2}} \\
 &\leq 
\frac{b}{a} \frac{ T
 \big( {\E}  \big[ \|F(\xi, Z^0)\|^2 \big] \big) ^{1/2}(1+2a^{-1}bLT)^{n} e^{a^{-1}bLT}  e^{n/2}}{n^{n/2}} \\
 &=ba^{-1}e^{a^{-1}bLT}T({\E}[ \|F(\xi, Z^0) \|^2])^{1/2}(1+2a^{-1}bLT)^{n}e^{n/2}n^{-n/2}\\
 &=C(1+2a^{-1}bLT)^{n}e^{n/2}n^{-n/2}.
 \end{split}
\end{equation}
 \cref{Lemma2.4} (applied for every $n \in \N$ with $T \curvearrowleft T$, $L \curvearrowleft  a^{-1}bL$, $C\curvearrowleft C$,
 $\alpha \curvearrowleft 3$,
 $\mathfrak{N} \curvearrowleft \mathfrak{N}$, 
 $\gamma_n \curvearrowleft (3n)^n$,
 $\epsilon_n \curvearrowleft \big( {\E} \big[\| X(T) - {\mathcal{X}}_{n,n}^{0}(T)\|^2\big] 
\big)^{1/2}$
 in the notation of \cref{Lemma2.4})  hence proves that
 for all $\varepsilon \in (0,1]$, $\delta \in (0, \infty)$ it holds that
\begin{equation} \label{thisforitem2}
 \sup_{n \in  [N_\varepsilon, \infty) \cap \N} \big( {\E} \big[\| X(T) - {\mathcal{X}}_{n,n}^{0}(T)\|^2\big] 
\big)^{1/2}
     \leq \varepsilon
     \end{equation} and
\begin{equation}
\sup_{n \in  [1,N_\varepsilon + \mathfrak{N}] \cap \N}  (3n)^n
\leq 
\sup_{n \in \N} 
  \left[ \frac{[3(n + \mathfrak{N}+1)]^{(n + \mathfrak{N}+1)}}{n^{n(1+\delta)}}
  e^{n(1+\delta)}
  (1+2a^{-1}bLT)^{n(2+2\delta)} \right] 
  \max \lbrace 1, C^{2+2\delta} \rbrace \varepsilon^{-(2+2\delta)}
    < \infty.
\end{equation}
This, \eqref{againRV}, \cref{lemma1},  \cref{l4.3},
and the fact that
for all $K \in \N$, $\beta \in [1,\infty)$ it holds that
$\sup_{n \in  [1,K] \cap \N} (\beta n)^n= (\beta K)^K$
ensure that for all $\varepsilon \in (0,1]$, $\delta \in (0, \infty)$ it holds that
\begin{equation}
\begin{split}
    \sum_{n=1}^{N_\varepsilon+\mathfrak{N}} RV_{n,n} 
    & \leq \sum_{n=1}^{N_\varepsilon+\mathfrak{N}} (3n)^n \leq 2[3(N_\varepsilon+\mathfrak{N})]^{(N_\varepsilon+\mathfrak{N})}=
    2 \sup_{n \in  [1,N_\varepsilon+\mathfrak{N}] \cap \N} (3 n)^n \\
    &\leq 2
    \sup_{n \in \N} 
  \left[ \frac{[3(n + \mathfrak{N}+1)]^{(n + \mathfrak{N}+1)}}{n^{n(1+\delta)}}
  e^{n(1+\delta)}
  (1+2a^{-1}bLT)^{n(2+2\delta)} \right]
  \max \lbrace 1, C^{2+2\delta} \rbrace \varepsilon^{-(2+2\delta)} < \infty.
    \end{split}
\end{equation}
This and \eqref{thisforitem2} establish item~\eqref{it:THEOREM:item2}. 
The proof of \cref{THEOREM} is thus completed.
\end{proof}

\bibliographystyle{acm}
\bibliography{References}

\end{document}